\documentclass[11pt]{mcom-l}
\usepackage{amssymb,amsfonts,amsmath,epsfig,hyperref}
\usepackage[margin=1.25in]{geometry}
\usepackage{mathrsfs}
\usepackage{color}     
\newtheorem{theorem}{Theorem}[section]
\newtheorem{proposition}{Proposition}[section]
\newtheorem{lemma}[theorem]{Lemma}
\theoremstyle{remark}
\newtheorem{remark}{\bf Remark}[section]
\renewcommand{\theequation}{\thesection.\arabic{equation}}
\allowdisplaybreaks

\def\nu{n}

\def\d{{\mathrm d}}
\def\D{{\underline{D}}}

\def\R{{\mathbb R}}

\newcommand\bfe{{\mathbf e}}

\newcommand\bfw{{\mathbf w}}
\newcommand\bfx{{\mathbf x}}

\newcommand\bfz{{\mathbf z}}
\newcommand\bfA{{\mathbf A}}

\newcommand\bfM{{\mathbf M}}

\renewcommand\d{\text{d}}

\newcommand{\laplace}{\Delta}

\newcommand{\nb}{\nabla}

\newcommand{\spn}{\textnormal{span}}

\def \to {\rightarrow}

\newcommand{\dof}{N}

\title{ 
Convergence of Dziuk's semidiscrete finite element method for mean curvature flow of closed surfaces with high-order finite elements
} 

\author{Buyang Li}
\address{Department of Applied Mathematics, 
The Hong Kong Polytechnic University, Hung Hom, Hong Kong.} 
\email {\href{mailto:buyang.li@polyu.edu.hk}{buyang.li{\it @}polyu.edu.hk}}

\begin{document}

\maketitle

\begin{abstract}
\small  
Dziuk's surface finite element method for mean curvature flow has had significant impact on the development of parametric and evolving surface finite element methods for surface evolution equations and curvature flows. However, the convergence of Dziuk's surface finite element method for mean curvature flow of closed surfaces  still remains open since it was proposed in 1990. 
In this article, we prove convergence of Dziuk's semidiscrete surface finite element method with high-order finite elements for mean curvature flow of closed surfaces. The proof utilizes the matrix-vector formulation of evolving surface finite element methods and a monotone structure of the nonlinear discrete surface Laplacian proved in this paper. 

\end{abstract}



\pagestyle{myheadings}
\markboth{}{}

\section{Introduction}\label{sec:intro}

We consider the evolution of a closed surface under mean curvature flow, moving with velocity $v=-Hn$, where $H$ and $n$ are the mean curvature and outward unit normal vector of the surface. The surface at time $t\in[0,T]$ can be described by 
$$ 
\Gamma(t)=\Gamma[X(\cdot,t)] = \{X(p,t): p \in \Gamma^0 \},\quad t\in[0,T] ,
$$
as the image of a flow map $X:\Gamma^0\times[0,T] \rightarrow \R^3$, which is a smooth embedding at every time $t\in[0,T]$ from a given closed initial surface $\Gamma^0$ into $\R^3$, satisfying the following geometric evolution equation: 
\begin{align}\label{PDE-X-P}
\left\{\begin{aligned}
&\partial_t X(p,t) =
(\Delta_{\Gamma[X(\cdot,t)]}{\rm id})\circ X(p,t) &&\mbox{for}\,\,\, p \in \Gamma^0\,\,\,\mbox{and}\,\,\,t\in(0,T], \\[5pt] 
&X(p,0)= p &&\mbox{for}\,\,\, p\in \Gamma^0 ,
\end{aligned}
\right.
\end{align}
where $\Delta_{\Gamma[X(\cdot,t)]}$ denotes the Laplace--Beltrami operator on the surface $\Gamma[X(\cdot,t)]$, and ${\rm id}$ is the identity function satisfying ${\rm id}(x)=x$ for all $x\in\R^3$. 

Numerical approximation to mean curvature flow by parametric finite element method was first considered by Dziuk \cite{Dziuk90} in 1990. The method determines the parametrization of the unknown surface by solving partial differential equations on a surface using the surface finite element method (FEM). 
The evolution of the nodes determines the approximate evolving surface. This idea has had significant influence on the development of surface FEMs for many different types of geometric evolution equations, and was systematically developed to the evolving surface FEMs in \cite{DziukElliott_ESFEM}. 

However, proving convergence of Dziuk's method for mean curvature flow of closed surfaces remains still open. For curve shortening flow, convergence of semidiscrete FEM was proved in \cite{Dziuk94}; convergence of nonlinearly implicit and linearly implicit FEMs were proved in \cite{2006-Rusu} and \cite{Li-2020-SINUM}, respectively. Convergence of non-parametric FEMs for mean curvature flow of graph surfaces was proved by Deckelnick \& Dziuk \cite{deckelnick1995convergence,deckelnick2000error-graph}, but the analysis cannot be extended to closed surfaces. 

Many other techniques were also developed for approximating mean curvature flow. For example, Deckelnick \& Dziuk \cite{DeckelnickDziuk} has introduced an artificial tangential velocity to reformulate curve shortening flow into a non-divergence form; Barrett, Garcke \& N{\"u}rnberg introduced a parametric FEM based on a different variational formulation  \cite{BGN2008} and a parametric FEM based on choosing different test functions \cite{Barrett-2010-4}; Elliott and Fritz \cite{Elliott-Fritz-2017} introduced DeTurck's trick of re-parametrization into the computation of mean curvature flow, leading to a non-degenerate parabolic system in a non-divergence form, which generalizes the reformulation of Deckelnick \& Dziuk in \cite{DeckelnickDziuk}. 

For all the methods mentioned above, convergence of semi- and fully discrete FEMs for mean curvature flow of closed surfaces remains open. Convergence of semidiscrete FEMs was proved for curve shortening flow in \cite{DeckelnickDziuk,Elliott-Fritz-2017}, for anistropic curve shortening flow in \cite{Dziuk1999,Pozzi2007}, for curve shortening flow coupled with reaction--diffusion in \cite{Barrett-Deckelnick-Styles-2017,Pozzi-Stinner-2017}, and for mean curvature flow of axisymmetric surfaces in \cite{2019-Barrett-Deckelnick-Nurnberg} based on DeTurck's trick. The only convergence result of surface FEMs for mean curvature flow of closed surfaces was in \cite{Kovacs-Li-Lubich-2019} for an equivalent system of equations governing the evolution of normal vector and mean curvature, instead of for the original equation \eqref{PDE-X-P} used by Dziuk \cite{Dziuk90} and many others. 

In this paper, we prove convergence of Dziuk's semidiscrete FEM for mean curvature flow of closed surfaces for sufficiently high-order finite elements. Our proof utilizes two ideas, i.e, the matrix-vector formulation of the evolving surface FEM and the monotone structure of the finite element discrete operator associated to $-\Delta_{\Gamma[X]}{\rm id} \circ X $. The matrix-vector formulation was used in \cite{KLLP2017} in analysis of convergence of evolving surface FEMs for solution-driven surfaces; the monotone structure of the nonlinear finite element discrete operator associated to $-\Delta_{\Gamma[X]}{\rm id} \circ X $ was used in \cite{Li-2020-SINUM} for analysis of curve shortening flow. 

In the following, we briefly explain the two ideas in proving convergence of Dziuk's semidiscrete FEM for mean curvature flow of closed surfaces. 

Let $\bfx^0 =(p_1,\dots,p_N)^T$ be the vector that collects all nodes $p_j\in\Gamma^0$, $j=1,\dots,\dof,$ in a triangulation of the initial surface $\Gamma^0$ (with finite elements of degree $k$). The nodal vector $\bfx^0 $ defines an approximate surface $\Gamma_h^0$ that interpolates $\Gamma^0$ at the nodes $p_j$. 
We evolve the vector $\bfx^0$ in time and denote its position at time $t$ by $\bfx(t)$, which  determines the approximate surface $\Gamma_h[\bfx (t)]$ to mean curvature flow and  satisfies an ordinary differential equation in the matrix-vector form (see Section \ref{section:result} for details)
\begin{align}\label{FEM0}
{\bf M}(\bfx )\dot \bfx  + {\bf A}(\bfx )\bfx  = {\bf 0} ,
\end{align}
with initial value $\bfx(0)=\bfx^0$, 
where ${\bf M}(\bfx )$ and ${\bf A}(\bfx )$ are the mass and stiffness matrices on the surface $\Gamma_h[\bfx]$. 
Equation \eqref{FEM0} is the matrix-vector formulation of Dziuk's semidiscrete FEM. Correspondingly, Dziuk's linearly implicit parametric FEM in \cite{Dziuk90} is equivalent to the following linearly implicit Euler method for \eqref{FEM0}: 
\begin{align}\label{linearly-implicit-FEM}
{\bf M}(\bfx ^{n-1})\frac{\bfx ^n-\bfx ^{n-1}}{\tau} + {\bf A}(\bfx ^{n-1})\bfx ^n = {\bf 0} ,
\end{align}
where $\tau$ denotes the stepsize of time discretization. 

As mentioned in \cite{Barrett-Deckelnick-Styles-2017,Li-2020-SINUM}, the main difficulty of numerical analysis for mean curvature flow \eqref{PDE-X-P} is the lack of full parabolicity, namely there does not exist a positive constant $\lambda$ satisfying  
\begin{align}\label{parabolicity} 
- ( \Delta_{\Gamma[X]}{\rm id}  \circ X - \Delta_{\Gamma[Y]}{\rm id} \circ Y)\cdot (X-Y) 
\ge \lambda |\nabla_{\Gamma^0}(X-Y)|^2 , 
\end{align}
even if the two flow maps $X$ and $Y$ are smooth and sufficiently close to each other. Similarly, if we denote by $\Gamma_h[\bfx^*]$ the interpolated surface of exact surface $\Gamma$, and denote by $\|\cdot\|_{{\bf A}(\bfx^*)}$ the discrete $H^1$ semi-norm on $\Gamma_h[\bfx ^*]$ defined by
$$
\|{\bf v}\|_{{\bf A}(\bfx ^*)}^2
:= {\bf A}(\bfx ^*){\bf v}\cdot{\bf v} 
= \int_{\Gamma_h[\bfx ^*]} \nabla_{\Gamma_h[\bfx ^*]}v_h\cdot \nabla_{\Gamma_h[\bfx ^*]}v_h ,
$$
with $v_h$ denoting the finite element function on the surface $\Gamma_h[\bfx^*]$ with nodal vector ${\bf v}$, then there does not exist a positive constant $\lambda$ satisfying 
\begin{align}\label{parabolicity-FEM} 
\big({\bf A}(\bfx )\bfx -{\bf A}(\bfx ^*)\bfx ^*\big)\cdot(\bfx -\bfx ^*)  
\ge \lambda \|\bfx -\bfx ^*\|_{{\bf A}(\bfx ^*)} ^2 , 
\end{align}
even if the two vectors $\bfx $ and $\bfx ^*$ are sufficiently close to each other. This is the main difficulty in analysis of Dziuk's semidiscrete FEM for mean curvature flow of closed surfaces.

We overcome this difficulty by showing the following identity (a monotone structure): 
\begin{align}\label{monotone}
\big({\bf A}(\bfx )\bfx -{\bf A}(\bfx ^*)\bfx ^*\big)\cdot(\bfx -\bfx ^*) 
&= \int_0^1 \int_{\Gamma_h^\theta} |(\nabla_{\Gamma_h^\theta} e_h^\theta)\hat n_h^\theta|^2 \d\theta ,
\end{align}
where $e^\theta_h$ is the finite element function with nodal vector 
$$\bfe=\bfx -\bfx ^*$$ on the intermediate finite element surface $\Gamma_h^\theta=(1-\theta)\Gamma_h[\bfx ^*]+\theta\Gamma_h[\bfx ]$, and $\hat n_h^\theta$ is the unit normal vector on $\Gamma_h^\theta$. The identity \eqref{monotone} can be used to control the $H^1$ semi-norm of the {\it normal component} of the error. 
It was known for closed curves and was used to analyze convergence of Dziuk's linearly implicit FEM for curve shortening flow in \cite{Li-2020-SINUM}. We extend this approach to mean curvature flow of closed surfaces using the matrix-vector technique. 

In addition to \eqref{monotone}, we also show that 
\begin{align}\label{mass-trick}
({\bf M}(\bfx )\dot\bfx -{\bf M}(\bfx ^*)\dot \bfx ^*)  \cdot (\bfx-\bfx^*) 
\ge &
\frac12 \frac{\d}{\d t} \|\bfe\|_{\bfM(\bfx)}^2
-c\epsilon^{-1}\|\bfe\|_{\bfM(\bfx)}^2 \notag \\
&
-
\epsilon \int_0^1 \int_{\Gamma_h^\theta} |(\nabla_{\Gamma_h^\theta} e_h^\theta) \hat n_h^\theta|^2 \d\theta , 
\end{align} 
where $\epsilon$ can be arbitrarily small and $\|\bfe\|_{\bfM(\bfx)}$ is the discrete $L^2$ norm on the surface $\Gamma_h[\bfx]$, defined by 
$$
\|\bfe\|_{\bfM(\bfx)}^2
=\int_{\Gamma_h[\bfx]} 
|e_h|^2 ,
$$ 
with $e_h$ being the finite element function on the surface $\Gamma_h[\bfx]$ with nodal vector $\bfe$. 
Hence, the last term in \eqref{mass-trick} can be absorbed by \eqref{monotone} in the error estimation, and Gronwall's inequality can be applied to yield an error estimate. 

To illustrate the idea clearly without complicating the problem, we focus on Dziuk's semidiscrete FEM (instead of fully discrete FEMs). 
As we shall see, high-order finite elements of polynomial degree $k\ge 6$ are needed to bound the nonlinear terms in the error estimation, though the computations in \cite{Dziuk90} seem to work well with lower-order finite elements.

In the next section, we present rigorous description of the matrix-vector formulation of Dziuk's semidiscrete FEM, and present the main theorem of this paper. The proof of the main theorem is presented in Section \ref{section:proof}.

\section{The main result}\label{section:result}

\subsection{Basic notions and notation}
\label{subsection: basic notions}

If $u(\cdot,t)$ is a function defined on the surface $\Gamma(t)=\Gamma[X(\cdot,t)]$ for $t\in[0,T]$, then the material derivative of $u$ with respect to the parametrization $X$ is defined as 
$$
\partial_t^{\bullet}u(x,t) = \frac \d{\d t} \,u(X(p,t),t) \quad\hbox{ for } \ x=X(p,t) \in \Gamma(t) .
$$
On any regular surface $\Gamma\subset\R^3$, for any function $u:\Gamma\to\R$ we denote by $\nabla_{\Gamma}u:\Gamma\to\R^3$ the surface tangential gradient as a 3-dimensional column vector. For a vector-valued function $u=(u_1,u_2,u_3)^T:\Gamma\to\R^3$, we define 
$\nabla_{\Gamma}u=
(\nabla_{\Gamma}u_1,
 \nabla_{\Gamma}u_2,
\nabla_{\Gamma}u_3)$, where each $\nabla_{\Gamma}u_j$ is a 3-dimensional column vector. 
We denote by $\nabla_{\Gamma} \cdot f$ the surface divergence of a vector field $f$ on $\Gamma$, and by $\laplace_{\Gamma} u=\nabla_{\Gamma}\cdot \nabla_{\Gamma}u$ the Laplace--Beltrami operator applied to $u$; see \cite{DeckelnickDE2005} or \cite[Appendix~A]{Ecker2012} for these notions. 


\subsection{Triangulation}
\label{subsection: triangulation}

The given smooth initial surface $\Gamma^0$ is partitioned into an admissible family of shape-regular and quasi-uniform triangulations $\mathcal{T}_h$ with finite elements of degree $k$ and mesh size $h$; see \cite{DziukElliott_ESFEM,Demlow2009} for the notion of admissible family of triangulations. For a fixed triangulation with mesh size $h$, we denote by $\bfx^0=(p_1,\dots,p_{\dof})^T$ the vector that collects all nodes $p_j\in\Gamma^0$, $j=1,\dots,\dof,$ in the triangulation of $\Gamma^0$ by finite elements of degree $k$. The nodal vector $\bfx^0 $ defines an approximate surface $\Gamma_h^0$ that interpolates $\Gamma^0$ at the nodes $p_j$. 

We consider the evolution of the nodal vector $\bfx=(x_1,\cdots,x_N)^T$ and denote its value at time $t$ by $\bfx(t)$, with initial condition $\bfx (0)=\bfx ^0$. By piecewise polynomial interpolation on the plane reference triangle that corresponds to every curved triangle of the triangulation, the nodal vector $\bfx(t)$ defines a closed surface denoted by 
$$\Gamma_h(t)=\Gamma_h[\bfx(t)].$$
There exists a unique finite element function $X_h(\cdot,t)$ of polynomial degree $k$ defined on the surface $\Gamma_h[\bfx^0]$ satisfying 
$$
X_h(p_j,t)=x_j(t)\quad\mbox{for}\quad j=1,\dots,\dof.
$$ 
This is the discrete flow map, which maps the initial surface $\Gamma_h[\bfx^0]$ to $\Gamma_h[\bfx(t)]$.  
If $w(\cdot,t)$ is a function defined on $\Gamma_h[\bfx(t)]$ for $t\in[0,T]$, then the material derivative $\partial_{t,h}^\bullet w$ on $\Gamma_h[\bfx(t)]$ with respect to the discrete flow map $X_h$ is defined by 
$$
\partial_{t,h}^{\bullet} w(x,t)
= \frac{\d}{\d t} w(X_h(p,t),t) \quad\mbox{for}\,\,\, x=X_h(p,t) \in \Gamma_h[\bfx(t)].
$$

\subsection{Finite element spaces}
\label{subsection: FE space}

The globally continuous finite element basis functions on the surface $\Gamma_h[\bfx]$ are denoted by 
$$
	\phi_i[\bfx]:\Gamma_h[\bfx]\to\R, \qquad i=1,\dotsc,\dof,
$$
which satisfy 
\begin{equation*}
	\phi_i[\bfx](x_j) = \delta_{ij} \quad  \text{ for all } i,j = 1,  \dotsc, \dof .
\end{equation*}
The pullback of $\phi_i[\bfx]$ from any curved triangle on $\Gamma_h[\bfx]$ to the reference plane triangle is a polynomial of degree $k$. 
It is known that the basis functions $\phi_j[\bfx(t)] $, $j=1,\dots,\dof$, have the following transport property (see \cite{DziukElliott_ESFEM})
\begin{align}\label{transport-phi}
\partial_{t,h}^{\bullet} \phi_j[\bfx(t)] = 0 \quad\mbox{on}\,\,\,\Gamma_h[\bfx(t)],\,\,\, j=1,\dots,N.
\end{align}
The finite element space on the surface $\Gamma_h[\bfx]$ is defined as 
\begin{equation*}
	S_h(\Gamma_h[\bfx]) =\spn\bigg\{ \sum_{j=1}^N v_j \phi_j[\bfx]: 
	v_j\in \R^3 \bigg\} , 
\end{equation*}
where each $v_j$ is a $3$-dimensional column vector. 
%

\subsection{Interpolated surface and lift onto the exact surface}
\label{subsection: lift}

In order to compare functions on the exact surface $\Gamma[X(\cdot,t)]$ with functions on the approximate surface $\Gamma_h[\bfx(t)]$, we introduce the interpolated surface $\Gamma_h[\bfx^*(t)]$, where $\bfx^*(t)$ denotes the nodal vector collecting the nodes $x_j^*(t)=X(p_j,t)$, $j=1,\dots,\dof$, moving along with the exact surface. 

For any point $x\in \Gamma_h[\bfx^*(t)]$ there exists a unique lifted point $x^l\in\Gamma[X(\cdot,t)]$, which was defined for linear and higher-order surface approximations in \cite{Dziuk88} and \cite{Demlow2009}, respectively. The lift operator is one-to-one and onto. As a result, any function $w$ on $\Gamma_h[\bfx^*]$ can be lifted to a function $w^l$ on $\Gamma$, defined as. 
$$
w^l(x^l)=w(x) . 
$$
Let $\delta_h(x)$ be the quotient between the continuous and interpolated surface measures, i.e., $\d A(x^l)=\delta_h(x)\d A_h(x)$. Then the following inequality holds (cf. \cite[Lemma 5.2]{Kovacs2018}):
\begin{align}\label{delta_h-1}
\|1-\delta_h\|_{L^\infty(\Gamma_h^*)}  \le
ch^{k+1} . 
\end{align}


If we denote by $I_h:C(\Gamma[X(\cdot,t)])\rightarrow S_h(\Gamma_h[\bfx^*(t)])$ the standard Lagrange interpolation operator, then the lifted Lagrange interpolation $(I_hv)^l$ approximates a function $v$ on $\Gamma[X(\cdot,t)]$ with optimal-order accuracy (cf. \cite[Proposition 2.7]{Demlow2009}), i.e.,
\begin{align}\label{interpl-error}
&\|v-(I_hv)^l\|_{L^2(\Gamma[X(\cdot,t)])} \le C\|v\|_{H^{k+1}(\Gamma[X(\cdot,t)])} h^{k+1} .
\end{align}
We denote by $\hat n_h^*$ the normal vector on $\Gamma_h[\bfx^*(t)]$ and denote by $\hat n_h^{*,l}$ its lift onto $\Gamma[X(\cdot,t)]$. 
Then $\hat n_h^{*,l}$ approximates the normal vector $n$ on $\Gamma[X(\cdot,t)]$ with the following accuracy (cf. \cite[Propositions 2.3]{Demlow2009}): 
\begin{align}\label{normal-error}
&\|\hat n_h^{*,l} - n \|_{L^\infty(\Gamma[X(\cdot,t)])} \le C h^{k} . 
\end{align}

\subsection{The main result}
\label{subsection: main result}

The mean curvature flow equation \eqref{PDE-X-P} can be equivalently written as 
\begin{align}\label{PDE-X-P-id}
\left\{\begin{aligned}
&\partial_t^{\bullet}{\rm id} = \Delta_{\Gamma[X(\cdot,t)]}{\rm id}  &&\mbox{on}\,\,\, \Gamma[X(\cdot,t)],\,\,\,\mbox{for}\,\,\, t\in(0,T], \\[5pt] 
&\Gamma[X(\cdot,0)]= \Gamma^0 .
\end{aligned}
\right.
\end{align}
Correspondingly, the semidiscrete evolving surface FEM for mean curvature flow is to find a nodal vector $\bfx(t)$, $t\in[0,T]$, such that the corresponding approximate surface $\Gamma_h[\bfx(t)]$ satisfies the following weak form:
\begin{align}\label{semidiscrete-FEM}
\left\{\begin{aligned}
&\int_{\Gamma_h[\bfx (t)]} \partial_t^{\bullet}{\rm id} \cdot v_h
+\int_{\Gamma_h[\bfx (t)]} \nabla_{\Gamma_h[\bfx (t)]}{\rm id} : \nabla_{\Gamma_h[\bfx (t)]}v_h
=0 \\[5pt]
& \hspace{100pt}
\quad\forall\, v_h\in S_h(\Gamma_h[\bfx (t)]), \,\,\,t\in(0,T],\\[5pt]
&\bfx(0)=\bfx^0 .
\end{aligned}
\right.
\end{align}

The mass matrix ${\bf M}(\bfx)$ and stiffness matrix ${\bf A}(\bfx )$ on the surface $\Gamma_h[\bfx ]$ consist of block components
\begin{align*}
{\bf M}_{ij}(\bfx ) = I_3 \int_{\Gamma_h[\bfx ]} \phi_i[\bfx ] \phi_j[\bfx ] 
\quad\mbox{and}\quad
{\bf A}_{ij}(\bfx ) = I_3  \int_{\Gamma_h[\bfx ]} \nabla_{\Gamma_h[\bfx ]}\phi_i[\bfx ] \cdot \nabla_{\Gamma_h[\bfx ]}\phi_j[\bfx ] ,
\end{align*}
for $i,j=1,\dots,\dof$, 
where $I_3$ is the $3\times 3$ identity matrix. 
Substituting 
$$
{\rm id} = \sum_{j=1}^{\dof} x_j(t)\phi_j[\bfx]
\quad\mbox{and}\quad
v_h=\phi_i[\bfx]
$$ 
into \eqref{semidiscrete-FEM} and using the transport property \eqref{transport-phi}, we obtain the following matrix-vector form of the semidiscrete FEM:
\begin{align}\label{FEM}
\left\{\begin{aligned}
&
{\bf M}(\bfx )\dot \bfx  + {\bf A}(\bfx )\bfx  = {\bf 0} ,
\\
&\bfx (0)=\bfx ^0 . 
\end{aligned}
\right.
\end{align}

The main result of this paper is the following theorem.\medskip

\begin{theorem}\label{MainTHM}
Consider the semidiscrete FEM \eqref{FEM} with finite elements of degree $k$. 
Suppose that the mean curvature flow problem \eqref{PDE-X-P} admits an exact solution $X$ that is sufficiently smooth on the time interval $t\in[0,T]$, and that the flow map $X(\cdot,t):\Gamma^0\to \Gamma[X(\cdot,t)]\subset\R^3$ is non-degenerate so that $\Gamma[X(\cdot,t)]$ is a regular surface for every $t\in[0,T]$. 
Then, there exists a constant $h_0 >0$ such that for all mesh sizes $h \leq h_0$ the following error bound holds when $k\ge 6$:
\begin{align}
&\max_{t\in[0,T]}\|X_h^l(\cdot,t) - X(\cdot,t)\|_{L^2(\Gamma^0)} \leq ch^{k-1} , 
\label{X_h-X} \\
&\max_{t\in[0,T]}\|\bfx(t) - \bfx^*(t)\|_{\infty} \leq ch^{k-2} ,
\label{x-xstar} 
\end{align}
where $X_h^l(\cdot,t)$ denotes the lift of the approximate flow map $X_h(\cdot,t)$ from $\Gamma_h[\bfx^0]$ onto $\Gamma^0$, and the constant $c$ is independent of $h$. 
\end{theorem}
\medskip


\section{Proof of Theorem \ref{MainTHM}}\label{section:proof}

Throughout, we denote by $c$ a generic positive constant that takes different values on different occurrences.

 \subsection{Preliminaries}

We denote by $\bfe=(e_1,\cdots,e_\dof)^T=\bfx-\bfx^*$ the vector consisting of the errors of numerical solutions at the nodes, and denote by 
$$
	e_h =\sum_{j=1}^\dof e_j \phi_j[\bfx^*] 
$$
the finite element error function on surface $\Gamma_h[\bfx^*]$. 

Let $t^*\in[0,T]$ be the maximal time such that the solution of \eqref{FEM} exists and the following inequality hold (with coefficient $1$):
\begin{equation}
\label{eq:assumed bounds - 2}
\begin{aligned}
&\| e_h(\cdot,t)\|_{L^2(\Gamma_h[\bfx^*(t)])} \leq h^{4}  &&\textrm{ for } \quad t\in[0,t^*].
\end{aligned}
\end{equation}
Since $e_h(\cdot,0)=0$, it follows that $t^*>0$, as the solution $\bfx$ of the ordinary differential equation \eqref{FEM} exists locally in time and is continuous in time. In the following, we prove the stated error bounds for $t\in[0,t^*]$. Then we show that  $t^*$ actually coincides with $T$.

The smoothness and non-degeneracy of the flow map $X(\cdot,t):\Gamma^0\to\Gamma(t)$ guarantees that it is locally close to an invertible linear transformation with bounded gradient uniformly with respect to $h$. Hence, it preserves the admissibility of grids with sufficiently small mesh width $h\le h_0$. This guarantees that the triangulations determined by the nodes $x_j^*(t)=X(p_j,t)$ remain admissible uniformly for $t\in[0,T]$ and $h\le h_0$, and the interpolated flow map $X_h^*(\cdot,t)$ and its inverse are bounded in $W^{1,\infty}(\Gamma_h[\bfx^0])$ (uniformly in $h$). 
Then \eqref{eq:assumed bounds - 2} implies, through inverse inequality, 
\begin{equation}
\label{eq:assumed bounds - W1infty}
\begin{aligned}
&\| e_h(\cdot,t)\|_{W^{1,\infty}(\Gamma_h[\bfx^*(t)])} \leq ch^2  &&\textrm{ for } \quad t\in[0,t^*] .
\end{aligned}
\end{equation}
\begin{remark}\upshape 
The powers of $h$ in \eqref{eq:assumed bounds - 2} and \eqref{eq:assumed bounds - W1infty} are needed to bound the nonlinear terms in the error estimation. For example, \eqref{eq:assumed bounds - 2} is used to prove the $W^{1,\infty}$ boundedness of the numerical velocity in \eqref{div-vh}, which is used in bounding the nonlinear term in \eqref{1st-term}; inequality \eqref{eq:assumed bounds - W1infty} is used in estimating the nonlinear terms in \eqref{estimate-J1}, \eqref{estimate-J6} and \eqref{estimate-J7}. 
The powers of $h$ in \eqref{eq:assumed bounds - 2} and \eqref{eq:assumed bounds - W1infty} require high-order finite elements of degree $k\ge 6$ in view of our error estimate \eqref{X_h-X} --- the power of $h$ in \eqref{X_h-X} should be strictly bigger than $4$ in order to absorb the constant $c$ in the derivation of  \eqref{eq:assumed bounds - 2}. This is done in \eqref{conclude-h4} for sufficiently small $h$.  
\end{remark}

Since $X_h(\cdot,t)=X_h^*(\cdot,t)+e_h(\cdot,t)\circ X_h^*(\cdot,t)$ and $X_h^*(\cdot,t)$ is bounded in $W^{1,\infty}(\Gamma_h[\bfx^0])$ (uniformly in $h$), the estimate above guarantees that the approximate flow map $X_h(\cdot,t):\Gamma_h[\bfx^0]\to\Gamma_h[\bfx(t)]$ and its inverse are bounded in $W^{1,\infty}(\Gamma_h[\bfx^0])$ uniformly with respect to $h$. Since deformation is the gradient of position, the boundedness of $X_h(\cdot,t)$ in $W^{1,\infty}(\Gamma_h[\bfx^0])$ (uniformly with respect to $h$) guarantees that the mesh on the approximate surface is not degenerate. Moreover, we can define an intermediate surface 
\begin{align}\label{def-x-theta}
\Gamma_h^{\theta}:=\Gamma_h[\bfx^{\theta}]
\quad\mbox{with nodal vector\,\, $\bfx^{\theta}=(1-\theta)\bfx^*+\theta\bfx.$} 
\end{align}
The estimate \eqref{eq:assumed bounds - W1infty} also guarantees that the intermediate surface $\Gamma_h^{\theta}$ is well defined with non-degenerate mesh, with 
$$
\Gamma_h^{1}=\Gamma_h[\bfx]
\quad\mbox{and}\quad
\Gamma_h^0=\Gamma_h^*=\Gamma_h[\bfx^*].
$$
The argument above is standard and was used in \cite{Kovacs-Li-Lubich-2019}. 



For any nodal vector $\bfw=(w_1,\dots,w_\dof)^T$ with $w_j\in\R^3$, we define a finite element function 
$$
w_h^{\theta}= \sum_{j=1}^N w_j \phi_j[\bfx^{\theta}] \in S_h(\Gamma_h^\theta)
$$
on the intermediate surface $\Gamma_h^{\theta}$. In particular, 
$$e_h^{\theta}=\sum_{j=1}^\dof e_j \phi_j[\bfx^{\theta}] $$ 
is the finite element error function on the surface $\Gamma_h^{\theta}$. As $\theta$ changes from $0$ to $1$, the surface $\Gamma_h^{\theta}$ moves with velocity $e_h^\theta=\sum_{j=1}^{\dof} e_j\phi_j[\bfx^\theta]$ (with respect to $\theta$). 
When $\theta=0$ we simply denote 
\begin{align}\label{def-eh-star-0}
e_h^*=e_h^{0} , 
\end{align}
which is a function on $\Gamma_h[\bfx^*]$. The lift of $e_h^*\in S_h[\Gamma_h^*]$ onto $\Gamma$ is denoted by $e_h^{*,l}$. 

On the intermediate surface $\Gamma_h^{\theta}$ we define the following discrete $L^2$ norm and $H^1$ semi-norm:
\begin{align} \label{M-L2} 
	&  \|\bfw\|_{\bfM(\bfx^{\theta})}^{2} = \bfw^T \bfM(\bfx^{\theta}) \bfw = \|w_h^{\theta}\|_{L^2(\Gamma_h^{\theta})}^2 , \\
	\label{A-H1}
	&  \|\bfw\|_{\bfA(\bfx^{\theta})}^{2} = \bfw^T \bfA(\bfx^{\theta}) \bfw = \|\nb_{\Gamma_h^{\theta}} w_h^{\theta}\|_{L^2(\Gamma_h^{\theta})}^2 .
\end{align}


\begin{lemma}
\label{lemma:matrix differences}   In the above setting, the following identities hold:
	\begin{align}
	\label{eq:matrix difference M}
&		\bfw^T (\bfM(\bfx)-\bfM(\bfx^*)) \bfz =\ \int_0^1 \int_{\Gamma_h^\theta} w_h^\theta (\nabla_{\Gamma_h^\theta} \cdot e_h^\theta) z_h^\theta \; \d\theta, \\[10pt]
\label{eq:matrix difference Ax}
&{\bf w}^T\big({\bf A}(\bfx )\bfx -{\bf A}(\bfx ^*)\bfx ^*\big) \notag \\
&=
\int_0^1 \int_{\Gamma_h^\theta} 
\bigg( \nabla_{\Gamma_h^\theta}  w_h^\theta : (D_{\Gamma_h^\theta} e_h^\theta)\nabla_{\Gamma_h^\theta} {\rm id} 
+ \nabla_{\Gamma_h^\theta}  w_h^\theta : \nabla_{\Gamma_h^\theta}  e_h^\theta \bigg) \; \d \theta ,
	\end{align}
	where $A:B ={\rm tr}(A^TB) $ for two $3\times 3$ matrices $A$ and $B$, and 
$$
D_{\Gamma_h^\theta} e_h^\theta =  \textnormal{tr}(E^\theta) I_3 - (E^\theta+(E^\theta)^T) 
	\quad\mbox{with}\quad
E^\theta=\nabla_{\Gamma_h^\theta} e_h^\theta .
$$
\end{lemma} 
\begin{proof}
Identity \eqref{eq:matrix difference M} was proved in \cite[Lemma 4.1]{KLLP2017}. Identity \eqref{eq:matrix difference Ax} can be proved as follows. 

Let ${\bf w}=(w_1,\cdots,w_\dof)^T$ and denote $w^\theta=\sum_{j=1}^{\dof} w_j\phi_j[\bfx^\theta]$ to be the finite element function on the surface $\Gamma_h^\theta$ with nodal vector ${\bf w}$, where $\bfx^\theta$ is defined in \eqref{def-x-theta}. As $\theta$ changes from $0$ to $1$, the surface $\Gamma_h^\theta$ moves with velocity $e_h^\theta=\sum_{j=1}^{\dof} e_j\phi_j[\bfx^\theta]$ with respect to $\theta$ and $\partial_\theta^{\bullet}w^\theta=0$. By using the fundamental theorem of calculus and the Leibniz formula, we have
\begin{align*}
 &\hspace{-8pt} {\bf w}^T\big({\bf A}(\bfx )\bfx -{\bf A}(\bfx ^*)\bfx ^*\big) \\
 =&\ \int_{\Gamma_h^1} \!\!\! \nabla_{\Gamma_h^1} w_h^1 : \nabla_{\Gamma_h^1} {\rm id}  
 - \int_{\Gamma_h^0} \!\!\! \nabla_{\Gamma_h^0} w_h^0 : \nabla_{\Gamma_h^0} {\rm id}  \\
    =& \int_0^1 \bigg(\frac{\d}{\d \theta} \int_{\Gamma_h^\theta} \nabla_{\Gamma_h^\theta} w_h^\theta  : \nabla_{\Gamma_h^\theta} {\rm id} \bigg) \d\theta \\
=&\ \int_0^1 \int_{\Gamma_h^\theta} 
\bigg( (\nabla_{\Gamma_h^\theta}  \cdot e_h^\theta) \nabla_{\Gamma_h^\theta}  w_h^\theta :\nabla_{\Gamma_h^\theta} {\rm id}  
- \nabla_{\Gamma_h^\theta}  w_h^\theta : (\nabla_{\Gamma_h^\theta} e_h^\theta+(\nabla_{\Gamma_h^\theta} e_h^\theta)^T)\nabla_{\Gamma_h^\theta} {\rm id} \\
&\hspace{50pt}
+\nabla_{\Gamma_h^\theta} \partial_\theta^{\bullet} w_h^\theta : \nabla_{\Gamma_h^\theta} {\rm id}  + \nabla_{\Gamma_h^\theta}  w_h^\theta : \nabla_{\Gamma_h^\theta} \partial_\theta^{\bullet}{\rm id}  \bigg) \; \d \theta ,
\end{align*}
where the last equality was essentially proved in \cite[eq (2.11)]{DziukElliott_ESFEM}. By using the notation $E^\theta$ and $D_{\Gamma_h^\theta} e_h^\theta$ in Lemma \ref{lemma:matrix differences}, and using the identities 
$$
\partial_\theta^{\bullet}w^\theta=0
\quad\mbox{and}\quad \partial_\theta^{\bullet}{\rm id}=e_h^\theta ,
$$ 
we obtain \eqref{eq:matrix difference Ax}. 
\hfill\end{proof}

%
The following lemma combines \cite[Lemmas 4.2 and 4.3]{KLLP2017} and \cite[Lemma 7.3]{Kovacs-Li-Lubich-2019}.\medskip

\begin{lemma} \label{lemma:theta-independence} 
In the above setting, if 
	$$
	\| \nabla_{\Gamma_h[\bfx^*]} e_h^* \|_{L^\infty(\Gamma_h[\bfx^*])} \le \tfrac12,
	$$
then for $0\le\theta\le 1$ and $1\le p \le \infty$ the finite element function 
	$$
	w_h^{\theta} =\sum_{j=1}^N w_j \phi_j[\bfx^{\theta}]\quad\mbox{on}\quad \Gamma_h^{\theta}
	$$ 
satisfies the following norm equivalence:
	\begin{align*}
		&\| w_h^{\theta} \|_{L^p(\Gamma_h^{\theta})} \leq c_p \, \|w_h^0 \|_{L^p(\Gamma_h[\bfx^*])} , 
		\\
		&\| \nabla_{\Gamma_h^{\theta}} w_h^{\theta} \|_{L^p(\Gamma_h^{\theta})} \le c_p \, \| \nabla_{\Gamma_h[\bfx^*]} w_h^0 \|_{L^p(\Gamma_h[\bfx^*])} , 
	\end{align*}
	where $c_p$ is an constant independent of $0\le \theta\le 1$ and $h$, with $c_\infty=2$.
\end{lemma} \medskip

For sufficiently small $h$, \eqref{eq:assumed bounds - W1infty} guarantees that 
$$
\| \nabla_{\Gamma_h[\bfx^*]} e_h^0 \|_{L^\infty(\Gamma_h[\bfx^*])}=\| \nabla_{\Gamma_h[\bfx^*]} e_h^* \|_{L^\infty(\Gamma_h[\bfx^*])}\le \frac14 . 
$$
Then Lemma \ref{lemma:theta-independence} implies that 
\begin{equation}\label{e-theta}
\| \nabla_{\Gamma_h^\theta} e_h^\theta \|_{L^\infty(\Gamma_h^\theta)} \le \frac12, \qquad 0\le\theta\le 1 . 
\end{equation}
By using this result in Lemma \ref{lemma:matrix differences}, together with the definition of the discrete $L^2$ and $H^1$ norms in \eqref{M-L2}--\eqref{A-H1}, we obtain the following result (as in (7.7) of \cite{Kovacs-Li-Lubich-2019}): 
\begin{equation}
\label{norm-equiv}
	\begin{aligned}
		&\text{The norms $\|\cdot\|_{\bfM(\bfx^\theta)}$ are $h$-uniformly equivalent for $0\le\theta\le 1$,}
		\\
		&\text{and so are the norms $\|\cdot\|_{\bfA(\bfx^\theta)}$.}
	\end{aligned}
\end{equation}

\subsection{The monotone structure}

Note that the interpolated nodal vector $\bfx^*$ satisfies equation \eqref{FEM} up to some defect ${\bf d}$, i.e., 
\begin{align}\label{FEM-x-star}
&{\bf M}(\bfx^*)\dot \bfx^* + {\bf A}(\bfx^*)\bfx^*
= {\bf M}(\bfx^*) {\bf d} ,
\end{align}
where the defect satisfies the following estimate (to be proved in Section \ref{section:defect}):
\begin{align}\label{defect-estimate}
\|{\bf d}\|_{\bfM(\bfx^*)} \le Ch^{k-1}. 
\end{align}
Subtracting \eqref{FEM-x-star} from \eqref{FEM}, we obtain the error equation
\begin{align}\label{FEM-Error-Eq}
&{\bf M}(\bfx )\dot \bfe 
+ {\bf A}(\bfx )\bfx  -{\bf A}(\bfx ^*)\bfx ^*
= - ({\bf M}(\bfx )-{\bf M}(\bfx ^*))\dot \bfx ^*  
- {\bf M}(\bfx^*) {\bf d} . 
\end{align} 

By using Lemma \ref{lemma:matrix differences}, we have 
\begin{align}\label{idea-surface-1}
&\big({\bf A}(\bfx )\bfx -{\bf A}(\bfx ^*)\bfx ^*\big)\cdot(\bfx -\bfx ^*) \notag \\
&=
\int_0^1 \int_{\Gamma_h^\theta} 
\bigg( \nabla_{\Gamma_h^\theta}  e_h^\theta : (D_{\Gamma_h^\theta} e_h^\theta)\nabla_{\Gamma_h^\theta} {\rm id} 
+ \nabla_{\Gamma_h^\theta}  e_h^\theta : \nabla_{\Gamma_h^\theta}  e_h^\theta \bigg) \; \d \theta \notag \\
&=
\int_0^1 \int_{\Gamma_h^\theta} 
\bigg( \nabla_{\Gamma_h^\theta} e_h^\theta : [ (D_{\Gamma_h^\theta} e_h^\theta)P^\theta  + \nabla_{\Gamma_h^\theta}  e_h^\theta ] \bigg) \; \d \theta  ,
\end{align}
where $D_{\Gamma_h^\theta} e_h^\theta =  \textnormal{tr}(E^\theta) I_3 - (E^\theta+(E^\theta)^T)$ is defined in Lemma \ref{lemma:matrix differences}, and we have used the identity 
$$
\nabla_{\Gamma_h^\theta} {\rm id} = I_3 - \hat n_h^\theta (\hat n_h^\theta)^T =: P^\theta , 
$$
with $\hat n_h^\theta$ denoting the unit normal vector on $\Gamma_h^\theta$ (thus $\hat n_h^\theta\notin S_h(\Gamma_h^\theta)$). 

Note that $P^\theta$ is a symmetric projection matrix satisfying
$$
P^\theta E^\theta = E^\theta,\,\,\,
(E^\theta)^TP^\theta = (E^\theta)^T
\,\,\,\mbox{and}\,\,\, 
{\rm tr}(P^\theta (E^\theta)^T)={\rm tr}((E^\theta)^TP^\theta)={\rm tr}(P^\theta E^\theta)={\rm tr}(E^\theta) .
$$
By using the properties above and the expression of $D_{\Gamma_h^\theta} e_h^\theta$, we furthermore reduce \eqref{idea-surface-1} to 
\begin{align}\label{idea-surface-2}
&\big({\bf A}(\bfx )\bfx -{\bf A}(\bfx ^*)\bfx ^*\big)\cdot(\bfx -\bfx ^*) \notag \\
&=
\int_0^1 \int_{\Gamma_h^\theta} 
\bigg[ {\rm tr} \Big( (E^\theta)^T ({\rm tr}(E^\theta) I_3  -  E^\theta - (E^\theta)^T )P^\theta \Big)+ {\rm tr}((E^\theta)^TE^\theta)\bigg] \; \d \theta \notag \\
&=
\int_0^1 \int_{\Gamma_h^\theta} 
\bigg[ {\rm tr} ({\rm tr}(E^\theta) (E^\theta)^TP^\theta  -  (E^\theta)^TE^\theta P^\theta - (E^\theta)^T(E^\theta)^TP^\theta ) + {\rm tr}((E^\theta)^TE^\theta)\bigg] \; \d \theta \notag \\
&=
\int_0^1 \int_{\Gamma_h^\theta} 
\bigg[  {\rm tr}(E^\theta)^2  - {\rm tr}(E^\theta E^\theta) + {\rm tr}((E^\theta)^TE^\theta(I-P^\theta))   \bigg] \; \d \theta .
\end{align}
Then we use the following lemma, of which the proof is presented in Section \ref{Section:Proof Lemma}.\medskip

\begin{lemma}\label{Lemma:key}
In the above setting, the following identity holds: 
\begin{align}\label{trE2-trE2}
\int_{\Gamma_h^\theta} \big[ {\rm tr}(E^\theta)^2  - {\rm tr}(E^\theta E^\theta) \big]  
= 0 .
\end{align}
\end{lemma}

By applying Lemma \ref{Lemma:key} to \eqref{idea-surface-2}, we obtain 
\begin{align}\label{idea-surface}
\big({\bf A}(\bfx )\bfx -{\bf A}(\bfx ^*)\bfx ^*\big)\cdot(\bfx -\bfx ^*) 
&=
\int_0^1 \int_{\Gamma_h^\theta} 
{\rm tr}((E^\theta)^TE^\theta(I-P^\theta))   \; \d \theta \notag \\
&=
\int_0^1 \int_{\Gamma_h^\theta} | (\nabla_{\Gamma_h^\theta} e_h^\theta) \hat n_h^\theta|^2 \d\theta . 
\end{align}
This is the key identity to be used in our error estimation. This identity reflects the monotone structure of the discrete nonlinear operator from $\bfx$ to ${\bf A}(\bfx )\bfx$. 

\subsection{Error estimation} 

Testing \eqref{FEM-Error-Eq} by $\bfe$ and using \eqref{idea-surface}, we obtain 
\begin{align}\label{FEM-Error-Eq-2}
&
{\bf M}(\bfx )\dot \bfe \cdot \bfe
+ \int_0^1 \int_{\Gamma_h^\theta} |(\nabla_{\Gamma_h^\theta} e_h^\theta) \hat n_h^\theta|^2 \d\theta \notag \\
&
= - ({\bf M}(\bfx )-{\bf M}(\bfx ^*))\dot \bfx ^*  \cdot \bfe 
- {\bf M}(\bfx^*) {\bf d}\cdot \bfe . 
\end{align} 
This can be equivalently formulated as 
\begin{align}\label{FEM-Error-Eq-3}
&\frac{\d}{\d t}\bigg( \frac12 {\bf M}(\bfx ) \bfe \cdot \bfe \bigg)
+ \int_0^1 \int_{\Gamma_h^\theta} |(\nabla_{\Gamma_h^\theta} e_h^\theta) \hat n_h^\theta|^2 \d\theta \notag \\
&= \frac12 \dot {\bf M}(\bfx ) \bfe \cdot \bfe 
- ({\bf M}(\bfx )-{\bf M}(\bfx ^*))\dot \bfx ^*  \cdot \bfe 
- {\bf M}(\bfx^*) {\bf d}\cdot \bfe .
\end{align} 

Let $v=-Hn$ be the velocity of the exact surface $\Gamma$, and let $v_j^*$ be the velocity of the exact surface at the $j$th interpolation node. We define  
$$
v_h^* =\sum_{j=1}^{\dof} v_j^*\phi_j[\bfx^*] = \sum_{j=1}^{\dof}\dot x_j^*\phi_j[\bfx^*] , 
$$ 
which is the interpolation of $v$ onto $S_h(\Gamma_h[\bfx^*])$. Let $v_h^{*,l}$ be the lift of $v_h^*$ onto the exact surface $\Gamma$, and denote 
$$
v_h^{*,\theta} = \sum_{j=1}^{\dof} v_j^*\phi_j[\bfx^\theta] = \sum_{j=1}^{\dof}\dot x_j^*\phi_j[\bfx^\theta] ,
$$
which is a finite element function on the surface $\Gamma_h^\theta$.

Let $v_h=\sum_{j=1}^{\dof} \dot x_j\phi_j[\bfx]$ be the velocity of the approximate surface $\Gamma_h[\bfx]$, and let 
$$
v_h^\theta=\sum_{j=1}^{\dof} \dot x_j\phi_j[\bfx^\theta] . 
$$
Then the nodal vector associated to the finite element function $v_h^0-v_h^*\in S_h(\Gamma_h[\bfx^*])$ is $\dot\bfe$, and by using the norm equivalence in Lemma \ref{lemma:theta-independence},
\begin{align} \label{div-vh}
\|\nabla_{\Gamma_h[\bfx]}\cdot v_h\|_{L^\infty(\Gamma_h[\bfx])}
\le &
c\|\nabla_{\Gamma_h[\bfx]}v_h\|_{L^\infty(\Gamma_h[\bfx])} \notag \\
\le &
c\|\nabla_{\Gamma_h[\bfx^*]}v_h^0\|_{L^\infty(\Gamma_h[\bfx^*])} \notag \\
\le &
c\|\nabla_{\Gamma_h[\bfx^*]}(v_h^0-v_h^*)\|_{L^\infty(\Gamma_h[\bfx^*])}
+c\|\nabla_{\Gamma_h[\bfx^*]}v_h^*\|_{L^\infty(\Gamma_h[\bfx^*])} \notag \\
\le &
ch^{-2}\|v_h^0-v_h^*\|_{L^2(\Gamma_h[\bfx^*])}
+c \quad\mbox{(inverse inequality)} \notag \\
= &
ch^{-2}\|\dot\bfe\|_{\bfM(\bfx^*)}
+c \notag \\
\le &
ch^{-4}\| \bfe \|_{\bfM(\bfx^*)} + ch^{k-3} 
+c \notag \\
\le & c ,
\end{align} 
where the last inequality uses \eqref{eq:assumed bounds - 2} and $k\ge 3$, and the second to last inequality can be proved as follows. Testing \eqref{FEM-Error-Eq} with ${\bf w}$, we obtain 
\begin{align}\label{M-dot-e-w}
\bfM(\bfx^*) \dot \bfe \cdot {\bf w} 
& = 
- ({\bf A}(\bfx )\bfx  -{\bf A}(\bfx ^*)\bfx ^*)\cdot {\bf w}
- ({\bf M}(\bfx )-{\bf M}(\bfx ^*))\dot \bfx ^* \cdot {\bf w} 
- {\bf M}(\bfx^*) {\bf d} \cdot {\bf w} . 
\end{align} 
By using Lemma \ref{lemma:matrix differences}, we have 
\begin{align}\label{dot-x-estimate-1}
&- ({\bf A}(\bfx )\bfx  -{\bf A}(\bfx ^*)\bfx ^*)\cdot {\bf w} \notag \\
& =
-\int_0^1 \int_{\Gamma_h^\theta} 
\bigg( \nabla_{\Gamma_h^\theta}  w_h^\theta : (D_{\Gamma_h^\theta} e_h^\theta)\nabla_{\Gamma_h^\theta} {\rm id} 
+ \nabla_{\Gamma_h^\theta}  w_h^\theta : \nabla_{\Gamma_h^\theta}  e_h^\theta \bigg) \; \d \theta \notag \\
& \le 
\int_0^1
c\|\nabla_{\Gamma_h^\theta}  w_h^\theta \|_{L^2(\Gamma_h^\theta)}\|\nabla_{\Gamma_h^\theta} e_h^\theta\|_{L^2(\Gamma_h^\theta)}   
\d\theta \notag \\
& \le 
\int_0^1 ch^{-2} \|w_h^\theta \|_{L^2(\Gamma_h^\theta)}\|e_h^\theta\|_{L^2(\Gamma_h^\theta)}   
\d\theta \notag \\
& =
ch^{-2} \|{\bf w}\|_{{\bf M}(\bfx ^*)}\|\bfe\|_{{\bf M}(\bfx ^*)}  .
\end{align} 
By denoting $\dot x_h^\theta=\sum_{j=1}^{\dof} \dot x_j^*\phi_j[\bfx^\theta]$, we have 
\begin{align}
- ({\bf M}(\bfx ) - {\bf M}(\bfx ^*) \dot \bfx ^*\cdot {\bf w} 
& =
-\int_0^1 \int_{\Gamma_h^\theta} 
(\nabla_{\Gamma_h^\theta} \cdot e_h^\theta) w_h^\theta \cdot \dot x_h^\theta  \, \d \theta \notag \\
& \le 
\int_0^1
c\|w_h^\theta \|_{L^2(\Gamma_h^\theta)}\|\nabla_{\Gamma_h^\theta} e_h^\theta\|_{L^2(\Gamma_h^\theta)}   
\d\theta \notag \\
& \le 
\int_0^1 ch^{-1} \|w_h^\theta \|_{L^2(\Gamma_h^\theta)}\|e_h^\theta\|_{L^2(\Gamma_h^\theta)}   
\d\theta \notag \\
& =
ch^{-1} \|{\bf w}\|_{{\bf M}(\bfx ^*)} \|\bfe\|_{{\bf M}(\bfx ^*)} . 
\end{align} 
By using the estimate \eqref{defect-estimate} for the defect ${\bf d}$, we have 
\begin{align}\label{dot-x-estimate-3}
{\bf M}(\bfx^*) {\bf d} \cdot {\bf w} 
&\le
c\|{\bf d}\|_{{\bf M}(\bfx ^*)} \|{\bf w}\|_{{\bf M}(\bfx ^*)} 
\le
ch^{k-1} \|{\bf w}\|_{{\bf M}(\bfx ^*)} .
\end{align} 
Substituting \eqref{dot-x-estimate-1}--\eqref{dot-x-estimate-3} into \eqref{M-dot-e-w} and choosing ${\bf w}=\dot\bfe$, we obtain
$$
\|\dot\bfe\|_{{\bf M}(\bfx ^*)}
\le
c(h^{k-1}+h^{-2}\|\bfe\|_{{\bf M}(\bfx ^*)}) .
$$
This proves the second to last inequality of \eqref{div-vh}. 

Recall that the finite element function on $\Gamma_h[\bfx]$ with the nodal vector $\bfe$ is denoted by $e_h^1$. 
By using \eqref{div-vh}, the first term on the right-hand side of \eqref{FEM-Error-Eq-3} can be estimated as follows:
\begin{align}\label{1st-term}
\frac12 \dot {\bf M}(\bfx ) \bfe \cdot \bfe 
=&
\frac12\int_{\Gamma_h[\bfx]} (\nabla_{\Gamma_h[\bfx]}\cdot v_h) e_h^1\cdot e_h^1
\quad\mbox{(this can be obtained from \cite[(2.9)]{DziukElliott_ESFEM})} \notag \\
\le&
c\|\nabla_{\Gamma_h[\bfx]}\cdot v_h\|_{L^\infty(\Gamma_h[\bfx])}
\|e_h^1\|_{L^2(\Gamma_h[\bfx])}^2 \notag\\
\le&
c \|\bfe\|_{\bfM(\bfx)}^2 \notag\\
\le&
c \|\bfe\|_{\bfM(\bfx^*)}^2 , 
\end{align} 
where the norm equivalence in \eqref{norm-equiv} is used. 

The third term on the right-hand side of \eqref{FEM-Error-Eq-3} satisfies 
\begin{align}\label{3rd-term}
- {\bf M}(\bfx^*) {\bf d}\cdot \bfe 
\le&
c\|{\bf d}\|_{\bfM(\bfx^*)} \|\bfe\|_{\bfM(\bfx^*)} 
\le
ch^{k-1} \|\bfe\|_{\bfM(\bfx^*)}  .
\end{align} 

We decompose the second term on the right-hand side of \eqref{FEM-Error-Eq-3} into several terms as follows: 
\begin{align}\label{2nd-term-1}
&-({\bf M}(\bfx )-{\bf M}(\bfx ^*))\dot \bfx ^*  \cdot \bfe \notag \\
&= 
-\int_0^1 \int_{\Gamma_h^\theta} (\nabla_{\Gamma_h^\theta} \cdot e_h^\theta) v_h^{*,\theta} \cdot e_h^\theta \; \d\theta \notag \\
&= 
-\int_0^1 
\bigg(\int_{\Gamma_h^\theta}  (\nabla_{\Gamma_h^\theta} \cdot e_h^\theta)v_h^{*,\theta}\cdot e_h^\theta - \int_{\Gamma_h^*}  (\nabla_{\Gamma_h^*} \cdot e_h^*) v_h^*\cdot e_h^* \bigg)\; \d\theta \notag \\
&\quad
-
\int_0^1 
\bigg(\int_{\Gamma_h^*}  (\nabla_{\Gamma_h^*} \cdot e_h^*) v_h^*\cdot e_h^* - \int_{\Gamma} (\nabla_{\Gamma_h^*} \cdot e_h^*)^l v_h^{*,l} \cdot e_h^{*,l} \bigg)\; \d\theta \notag \\
&\quad
-
\int_0^1 
\int_{\Gamma} \big[ (\nabla_{\Gamma_h^*} \cdot e_h^*)^l - \nabla_{\Gamma} \cdot e_h^{*,l} \big] v_h^{*,l} \cdot e_h^{*,l} \; \d\theta \notag \\
&\quad
-
\int_0^1 
\int_{\Gamma} (\nabla_{\Gamma} \cdot e_h^{*,l}) (v_h^{*,l} -v)\cdot e_h^{*,l}  \; \d\theta \notag \\
&\quad
+ 
\int_0^1 
\int_{\Gamma} (\nabla_{\Gamma} \cdot e_h^{*,l}) Hn \cdot e_h^{*,l}
\d\theta 
\quad\mbox{(we have substituted $v=-Hn$ here)} \notag \\
&=: J_1+J_2+J_3+J_4
+ \int_0^1 
\int_{\Gamma} (\nabla_{\Gamma} \cdot e_h^{*,l}) Hn \cdot e_h^{*,l} 
\d\theta . 
\end{align} 
The purpose of transforming from $\Gamma_h^\theta$ to $\Gamma$ (namely to be able to replace $v$ with $Hn$) is to perform integration by parts on the last term of \eqref{2nd-term-1}. This would yield $(\nabla_{\Gamma}  e_h^{*,l})n$, which is the only term that contains the partial derivative of $e_h^{*,l}$ on the right-hand side. 
The term $(\nabla_{\Gamma}  e_h^{*,l})n$ can be furthermore converted to $(\nabla_{\Gamma_h^\theta} e_h^\theta )\hat n_h^\theta  $ (which can be absorbed by the left-hand side of \eqref{FEM-Error-Eq-2}) after transforming $\Gamma$ back to $\Gamma_h^\theta$, as shown in the following estimates. 

The last term in \eqref{2nd-term-1} can be estimated as follows. Using the integration by parts formula (cf. \cite[Section 2.3]{DziukElliott_acta})
$$
\int_\Gamma  f \, \nabla_\Gamma \cdot \varphi =  \int_\Gamma f\, H n \cdot \varphi -  \int_\Gamma \nabla_\Gamma f \cdot \varphi  ,
$$
we have 
\begin{align*} 
&\int_0^1 
\int_{\Gamma} (\nabla_{\Gamma} \cdot e_h^{*,l}) Hn \cdot e_h^{*,l} 
\d\theta \\
&=\int_0^1 
\int_{\Gamma} |H n \cdot e_h^{*,l}|^2 
-\int_0^1 \int_{\Gamma} e_h^{*,l} \cdot \nabla_{\Gamma}(Hn \cdot e_h^{*,l} ) 
\d\theta \\
&= 
\int_0^1\int_{\Gamma} |H n \cdot e_h^{*,l}|^2 
\d\theta
- \int_0^1 
\int_{\Gamma} (e_h^{*,l} \cdot \nabla_{\Gamma} H) n \cdot e_h^{*,l} 
\d\theta 
- \int_0^1 
\int_{\Gamma} H e_h^{*,l} \cdot  (\nabla_{\Gamma} n ) e_h^{*,l} \, 
\d\theta \\
&\quad 
- \int_0^1 
\int_{\Gamma} He_h^{*,l} \cdot (\nabla_{\Gamma} e_h^{*,l}) \, n \,
\d\theta \\
&\le
c\|e_h^{*,l}\|_{L^2(\Gamma)}^2 
- \int_0^1 
\int_{\Gamma} He_h^{*,l} \cdot (\nabla_{\Gamma} e_h^{*,l}) \, n \,
\d\theta .
\end{align*} 
Recall that $\hat n_h^*$ denotes the normal vector on $\Gamma_h[\bfx^*]$ and $\hat n_h^{*,l}$ is the lift of $\hat n_h^*$ onto $\Gamma$. 
By introducing $H_h^*\in S_h(\Gamma_h[\bfx^*])$ to be the finite element interpolation of $H$, and denoting by $H_h^{*,l}$ the lift of $H_h^*$ to the surface $\Gamma$, the inequality above furthermore implies that 
\begin{align*} 
&\int_0^1 
\int_{\Gamma} (\nabla_{\Gamma} \cdot e_h^{*,l}) Hn \cdot e_h^{*,l} 
\d\theta \\
&\le 
c\|e_h^{*,l}\|_{L^2(\Gamma)}^2  
- \int_0^1 
\int_{\Gamma} (H-H_h^{*,l})e_h^{*,l} \cdot (\nabla_{\Gamma} e_h^{*,l}) \, n \, \d\theta \\
&\quad 
- \int_0^1 
\int_{\Gamma} H_h^{*,l} e_h^{*,l} \cdot (\nabla_{\Gamma} e_h^{*,l}) (n-\hat n_h^{*,l}) \, \d\theta 
- \int_0^1 
\int_{\Gamma} H_h^{*,l} e_h^{*,l} \cdot (\nabla_{\Gamma} e_h^{*,l}) \hat n_h^{*,l} \,
\d\theta  \\
&\le 
c\|e_h^{*,l}\|_{L^2(\Gamma)}^2 
+ ch^{k+1}\|e_h^{*,l}\|_{L^2(\Gamma)}\|\nabla_{\Gamma}e_h^{*,l}\|_{L^2(\Gamma)} \\
&\quad 
+ c h^{k}\|e_h^{*,l}\|_{L^2(\Gamma)}\|\nabla_{\Gamma}e_h^{*,l}\|_{L^2(\Gamma)} 
- \int_0^1 
\int_{\Gamma} H_h^{*,l} e_h^{*,l} \cdot (\nabla_{\Gamma} e_h^{*,l}) \hat n_h^{*,l} \,
\d\theta ,
\end{align*} 
where the last inequality uses the interpolation error estimate \eqref{interpl-error}--\eqref{normal-error}. 
By using the norm equivalence $\|e_h^{*,l}\|_{L^2(\Gamma)}\sim \|e_h^*\|_{L^2(\Gamma_h^*)}$ and $\|\nabla_{\Gamma}e_h^{*,l}\|_{L^2(\Gamma)}\sim \|\nabla_{\Gamma_h^*}e_h^*\|_{L^2(\Gamma_h^*)}$ in Lemma \ref{lemma:theta-independence}, and using the inverse inequality of finite element functions, we obtain from the above inequality
\begin{align*} 
&\int_0^1 
\int_{\Gamma} (\nabla_{\Gamma} \cdot e_h^{*,l}) Hn \cdot e_h^{*,l} 
\d\theta \\
&\le 
c\|e_h^*\|_{L^2(\Gamma_h^*)}^2 
-\int_0^1 
\int_{\Gamma} H_h^{*,l} e_h^{*,l} \cdot (\nabla_{\Gamma} e_h^{*,l}) \hat n_h^{*,l} \,
\d\theta \\
&=
c\|e_h^*\|_{L^2(\Gamma_h^*)}^2 
+ \int_0^1 
\bigg(
\int_{\Gamma_h^*}  H_h^* e_h^*\cdot (\nabla_{\Gamma_h^*} e_h^*) \hat n_h^*
-\int_{\Gamma} H_h^{*,l} e_h^{*,l} \cdot (\nabla_{\Gamma} e_h^{*,l} ) \hat n_h^{*,l}  \bigg) \,
\d\theta \\
&\quad
+ \int_0^1 
\bigg(
\int_{\Gamma_h^\theta}  H_h^{*,\theta} e_h^\theta\cdot (\nabla_{\Gamma_h^\theta} e_h^\theta ) \hat n_h^{*,\theta} 
-
\int_{\Gamma_h^*}  H_h^* e_h^*\cdot (\nabla_{\Gamma_h^*} e_h^*) \hat n_h^* 
\bigg) \, 
\d\theta \\
&\quad
+ \int_0^1 
\int_{\Gamma_h^\theta}  H_h^{*,\theta} e_h^\theta\cdot (\nabla_{\Gamma_h^\theta} e_h^\theta ) (\hat n_h^\theta-\hat n_h^{*,\theta}) \, 
\d\theta \\
&\quad 
-\int_0^1 
\int_{\Gamma_h^\theta}  H_h^{*,\theta} e_h^\theta\cdot (\nabla_{\Gamma_h^\theta} e_h^\theta )\hat n_h^\theta  \, 
\d\theta \\
&=
c\|e_h^*\|_{L^2(\Gamma_h^*)}^2 
+J_5+J_6+J_7+J_8 , 
\end{align*}
where $H_h^{*,\theta}$ is defined as the finite element function on $\Gamma_h^\theta$ with the same nodal vector as $H_h^{*}$. Substituting this into \eqref{2nd-term-1} yields 
\begin{align}\label{2nd-term-2}
-({\bf M}(\bfx )-{\bf M}(\bfx ^*))\dot \bfx ^*  \cdot \bfe  
&\le 
c\|e_h^*\|_{L^2(\Gamma_h^*)}^2 
+ \sum_{m=1}^8 J_m .
\end{align} 

\subsection{Estimation of $J_m$, $m=1,\dots,8$}
\begin{align}
J_1 
=& 
-\int_0^1 
\bigg(\int_{\Gamma_h^\theta}  (\nabla_{\Gamma_h^\theta} \cdot e_h^\theta)v_h^{*,\theta}\cdot e_h^\theta - \int_{\Gamma_h^*}  (\nabla_{\Gamma_h^*} \cdot e_h^*) v_h^*\cdot e_h^* \bigg)\; \d\theta \notag \\
=& 
-\int_0^1 
\int_0^{\theta} \bigg( \frac{\d}{\d\sigma}\int_{\Gamma_h^{\sigma}}  (\nabla_{\Gamma_h^{\sigma}} \cdot e_h^{\sigma})v_h^{*,\sigma}\cdot e_h^{\sigma} \bigg) \d\sigma \d\theta 
\quad\, \mbox{(Newton--Leibniz rule)}  \notag \\
=& 
-\int_0^1 (1-\sigma) \bigg( \frac{\d}{\d\sigma}\int_{\Gamma_h^{\sigma}}  (\nabla_{\Gamma_h^{\sigma}} \cdot e_h^{\sigma})v_h^{*,\sigma}\cdot e_h^{\sigma} \bigg) \d\sigma \quad \mbox{(order of integration is changed)} \notag \\
=& 
-\int_0^1 (1-\theta) \bigg( \frac{\d}{\d\theta}\int_{\Gamma_h^{\theta}}  (\nabla_{\Gamma_h^{\theta}} \cdot e_h^{\theta})v_h^{*,\theta}\cdot e_h^{\theta} \bigg) \d\theta \quad\,\,\, \mbox{($\sigma$ is changed to $\theta$)} \notag \\
=& 
- \int_0^1 
\bigg[ (1-\theta) \int_{\Gamma_h^{\theta}}  
\bigg( \partial_\theta^{\bullet}(\nabla_{\Gamma_h^{\theta}} \cdot e_h^{\theta})v_h^{*,\theta}\cdot e_h^{\theta} 
+ |\nabla_{\Gamma_h^{\theta}} \cdot e_h^{\theta}|^2 v_h^{*,\theta}\cdot e_h^{\theta}  
\bigg) \bigg] \d\theta .
\end{align} 
where the last inequality uses the properties $\partial_\theta^{\bullet}v_h^{*,\theta}=\partial_\theta^{\bullet}e_h^\theta=0$, and the fact that the surface $\Gamma_h^\theta$ moves with velocity $e_h^\theta$ with respect to $\theta$. 
By using the identity (cf. \cite[Lemma~2.6]{DziukKronerMuller}) 
\begin{align}
\label{eq:mat grad interchange}
\partial_\theta^\bullet \big( \nabla_{\Gamma_h^\theta} \cdot e_h^{\theta} \big) 
=
	\nabla_{\Gamma_h^\theta} \cdot \partial_\theta^\bullet e_h^{\theta} 
	- {\rm tr} \bigg[ \Big( \nabla_{\Gamma_h^\theta} e_h^\theta - \hat n_h^\theta (\hat n_h^\theta)^T (\nabla_{\Gamma_h^\theta} e_h^\theta)^T \Big) \nabla_{\Gamma_h^\theta}e_h^{\theta}\bigg] \notag \\
=
	- {\rm tr} \bigg[ \Big( \nabla_{\Gamma_h^\theta} e_h^\theta - \hat n_h^\theta (\hat n_h^\theta)^T (\nabla_{\Gamma_h^\theta} e_h^\theta)^T \Big) \nabla_{\Gamma_h^\theta}e_h^{\theta}\bigg] 
	\quad\mbox{(since $\partial_\theta^\bullet e_h^{\theta} =0$)} ,
\end{align}
we find that 
\begin{align}\label{estimate-J1}
J_1 
\le & 
\int_0^1 c \|\nabla_{\Gamma_h^{\theta}} e_h^{\theta}\|_{L^\infty(\Gamma_h^\theta)}
\|\nabla_{\Gamma_h^{\theta}} e_h^{\theta}\|_{L^2(\Gamma_h^\theta)}
\| e_h^{\theta}\|_{L^2(\Gamma_h^\theta)} \d\theta \notag\\
\le & 
\int_0^1 ch^2
\|\nabla_{\Gamma_h^{\theta}} e_h^{\theta}\|_{L^2(\Gamma_h^\theta)}
\| e_h^{\theta}\|_{L^2(\Gamma_h^\theta)} \d\theta 
\quad \mbox{(estimate \eqref{eq:assumed bounds - W1infty} is used)} \notag\\
\le & 
\int_0^1 ch
\|e_h^{\theta}\|_{L^2(\Gamma_h^\theta)}^2 \d\theta 
\quad \mbox{(inverse inequality)}  \notag\\
= & 
ch \|{\bf e}\|_{\bfM(\bfx^*)}^2. \quad\,\, \mbox{(norm equivalence \eqref{norm-equiv})} 
\end{align} 

Let $x^l$ denote the lift of $x\in \Gamma_h^*$ onto $\Gamma$. By using \eqref{delta_h-1} we have 
\begin{align}
J_2
&=
-
\bigg(\int_{\Gamma_h^*}  (\nabla_{\Gamma_h^*} \cdot e_h^*) v_h^*\cdot e_h^* - \int_{\Gamma} (\nabla_{\Gamma_h^*} \cdot e_h^*)^l v_h^{*,l} \cdot e_h^{*,l} \bigg) \notag \\ 
&=-
\int_{\Gamma_h^*} (1-\delta_h) (\nabla_{\Gamma_h^*} \cdot e_h^*) v_h^{*} \cdot e_h^{*} \notag \\
&\le
c\|1-\delta_h\|_{L^\infty(\Gamma_h^*)} 
\|\nabla_{\Gamma_h^*} \cdot e_h^*\|_{L^2(\Gamma_h^*)} 
\|v_h^{*}\|_{L^\infty(\Gamma_h^*)} \|e_h^{*}\|_{L^2(\Gamma_h^*)} \notag \\
&\le
ch^{k+1} 
\|\nabla_{\Gamma_h^*} \cdot e_h^*\|_{L^2(\Gamma_h^*)}  \|e_h^{*}\|_{L^2(\Gamma_h^*)} \notag \\
&\le
ch^{k} \|e_h^{*}\|_{L^2(\Gamma_h^*)}^2 \notag \\
&=
ch^{k} \|\bfe\|_{\bfM(\bfx^*)}^2 ,
\end{align}
where we have used inverse inequality in the second to last inequality.

For the exact surface $\Gamma=\Gamma(t)$, we denote by $d(x)$ the signed distance from $x$ to $\Gamma$, defined by 
$$
d(x)=
\left\{\begin{aligned}
&|x-x^l| &&\mbox{if}\,\,\, x\in\R^3\backslash\Omega,\\
&-|x-x^l| &&\mbox{if}\,\,\, x\in\Omega .
\end{aligned}
\right.
$$ 
Let $\mathcal{H}=\nabla_{\Gamma} n$ be the Weingarten matrix on $\Gamma$. Then the following identity holds (for example, see \cite[Remark 4.1]{Dziuk-Elliott-2013-MC}): 
\begin{align*}
\nabla_{\Gamma_{h}^{*}} e_h^{*}(x)=P_{h}(x)(I-d(x) \mathcal{H}(x^l)) \nabla_{\Gamma} e_h^{*,l}(x^l) ,
\end{align*}
where $P_h(x)=I_3-\hat n_h^*(x)\hat n_h^*(x)^T$, with $\hat n_h^*$ denoting the normal vector on $\Gamma_h^*$. Hence, denoting $P(x^l)=I_3-\hat n(x^l)\hat n(x^l)^T$, we have 
\begin{align}\label{lift-eh-ehl}
&|(\nabla_{\Gamma_{h}^{*}} e_h^{*})^l(x^l)-\nabla_{\Gamma} e_h^{*,l}(x^l)| \notag \\
&=
\Big|P_{h}(x)(I-d(x) \mathcal{H}(x^l)) \nabla_{\Gamma} e_h^{*,l}(x^l)
-P(x^l) \nabla_{\Gamma} e_h^{*,l}(x^l) \Big| \notag \\
&=
\Big| \big[ (P_{h}(x)-P(x^l))(I-d(x) \mathcal{H}(x^l))
-d(x) P(x^l)\mathcal{H}(x^l) \big]
\nabla_{\Gamma} e_h^{*,l}(x^l)  \Big| \notag \\
&\le 
(ch^k+ch^{k+1}) |\nabla_{\Gamma} e_h^{*,l}(x^l) | \notag \\
&\le ch^k |\nabla_{\Gamma} e_h^{*,l}(x^l) | . 
\end{align}
where the second to last inequality uses estimate \eqref{normal-error} in estimating $P_{h}(x)-P(x^l)$, and uses $|d(x)|\le ch^{k+1}$ (see \cite[Lemma 5.2]{Kovacs2018}). 
For sufficiently small $h$, the inequality above furthermore implies, via using the triangle inequality, 
\begin{align}\label{grad-Gamma-Gamma_h}
\|\nabla_{\Gamma} e_h^{*,l}\|_{L^2(\Gamma)}  
&\le 
c\|(\nabla_{\Gamma_h^*} e_h^{*})^l\|_{L^2(\Gamma)} 
\le 
c\|\nabla_{\Gamma_h^*} e_h^{*}\|_{L^2(\Gamma_h^*)} ,
\end{align}
where we have used the norm equivalence between $\|(\nabla_{\Gamma_{h}^*} e_h^{*})^l\|_{L^2(\Gamma)} $ and $\|\nabla_{\Gamma_{h}^*} e_h^{*}\|_{L^2(\Gamma_h^*)} $ as shown in Lemma \ref{lemma:theta-independence}. By using the two results above, we have 
\begin{align*}
J_3
&=
-\int_{\Gamma}  [(\nabla_{\Gamma_h^*} \cdot e_h^*)^l-\nabla_{\Gamma} \cdot e_h^{*,l}] v_h^{*,l} \cdot e_h^{*,l}  \\
&\le
ch^k\|\nabla_{\Gamma} e_h^{*,l}\| _{L^2(\Gamma)}
\| v_h^{*,l} \| _{L^\infty(\Gamma)}
\| e_h^{*,l} \| _{L^2(\Gamma)} \\
&\le
ch^{k-1} \|e_h^{*}\| _{L^2(\Gamma_h^*)}^2 \notag \\
&=
ch^{k-1} \|\bfe\|_{\bfM(\bfx^*)}^2 ,
\end{align*}
where we have used inverse inequality in the second to last inequality. 

Since the lifted Lagrange interpolation $v_h^{*,l} $ has optimal-order accuracy in approximating $v$, as shown in \eqref{interpl-error}, it follows that 
\begin{align*}
J_4
=-
\int_{\Gamma} (\nabla_{\Gamma} \cdot e_h^{*,l}) (v_h^{*,l} -v)\cdot e_h^{*,l} 
&\le 
c\|v_h^{*,l} -v\|_{L^\infty(\Gamma)} \|\nabla_{\Gamma} \cdot e_h^{*,l}\|_{L^2(\Gamma)}
\|e_h^{*,l}\|_{L^2(\Gamma)} \\
&\le 
ch^{k+1} \|\nabla_{\Gamma} \cdot e_h^{*,l}\|_{L^2(\Gamma)}
\|e_h^{*,l}\|_{L^2(\Gamma)} \\
&\le 
ch^{k+1} \|\nabla_{\Gamma} \cdot e_h^{*}\|_{L^2(\Gamma_h^*)}
\|e_h^{*}\|_{L^2(\Gamma_h^*)} \\
&\le 
ch^{k} \|e_h^{*}\|_{L^2(\Gamma_h^*)}^2 \\
&=
ch^{k} \|\bfe\|_{\bfM(\bfx^*)}^2 . 
\end{align*}

Recall that $H_h^*$ is the finite element interpolation of $H$ onto $\Gamma_h[\bfx^*]$ and $H_h^{*,l}$ is the lift of $H_h^*$ onto the surface $\Gamma$. 
By using \eqref{delta_h-1} and \eqref{lift-eh-ehl}, we can estimate $J_5$ similarly as $J_3$, i.e., 
\begin{align}
J_5
&=
\int_{\Gamma_h^*}  H_h^* e_h^*\cdot (\nabla_{\Gamma_h^*} e_h^*) \hat n_h^*
-\int_{\Gamma} H_h^{*,l} e_h^{*,l} \cdot (\nabla_{\Gamma} e_h^{*,l}) \hat n_h^{*,l} \notag \\
&=
\int_{\Gamma}  \delta_h^{-1} H_h^{*,l} e_h^{*,l} \cdot (\nabla_{\Gamma_h^*} e_h^*)^l \hat n_h^{*,l}
-\int_{\Gamma} H_h^{*,l} e_h^{*,l} \cdot (\nabla_{\Gamma} e_h^{*,l})\hat n_h^{*,l} \notag \\
&=
\int_{\Gamma} ( \delta_h^{-1} - 1 ) H_h^{*,l} e_h^{*,l} \cdot (\nabla_{\Gamma_h^*} e_h^*)^l  \hat n_h^{*,l}
+\int_{\Gamma} H_h^{*,l} e_h^{*,l} \cdot [(\nabla_{\Gamma_h^*} e_h^*)^l - \nabla_{\Gamma} e_h^{*,l} ]  \hat n_h^{*,l} \notag \\ 
&\le
ch^{k+1} \|e_h^{*,l}\|_{L^2(\Gamma)} \|\nabla_{\Gamma_h^*} e_h^{*,l}\|_{L^2(\Gamma)} 
+ch^k\|e_h^{*,l}\|_{L^2(\Gamma)} \|\nabla_{\Gamma_h^*} e_h^{*,l}\|_{L^2(\Gamma)} \notag \\
&\le
ch^{k+1} \|e_h^{*}\|_{L^2(\Gamma_h^*)} \|\nabla_{\Gamma_h^*} e_h^{*}\|_{L^2(\Gamma_h^*)} 
+ch^k\|e_h^{*}\|_{L^2(\Gamma_h^*)} \|\nabla_{\Gamma_h^*} e_h^{*}\|_{L^2(\Gamma_h^*)} \notag \\
&\le
ch^{k-1} \|e_h^{*}\|_{L^2(\Gamma_h^*)}^2 \notag \\
&=
ch^{k-1} \|\bfe\|_{\bfM(\bfx^*)}^2 ,
\end{align}
where we have used inverse inequality in the second to last inequality.

Recall that $H_h^{*,\theta}$ is finite element function on $\Gamma_h^\theta$ with the same nodal vector as the interpolated finite element function $H_h^{*}$. Since $e_h^0=e_h^*$, as defined in \eqref{def-eh-star-0}, it follows that 
\begin{align*}
J_6
&=
\int_0^1 
\bigg(
\int_{\Gamma_h^\theta}  H_h^{*,\theta} e_h^\theta\cdot (\nabla_{\Gamma_h^\theta} e_h^\theta ) \hat n_h^{*,\theta} 
-
\int_{\Gamma_h^*}  H_h^* e_h^*\cdot (\nabla_{\Gamma_h^*} e_h^* ) \hat n_h^* 
\bigg) \, 
\d\theta \\
&=
\int_0^1 
\int_0^{\theta} \frac{\d}{\d\sigma } \int_{\Gamma_h^\sigma}  H_h^{*,\sigma} e_h^\sigma\cdot (\nabla_{\Gamma_h^\sigma} e_h^\sigma)\hat n_h^{*,\sigma} \d\sigma 
\d\theta \\
&=
\int_0^1 
\int_{\sigma}^{1} \frac{\d}{\d\sigma } \int_{\Gamma_h^\sigma}  H_h^{*,\sigma} e_h^\sigma\cdot (\nabla_{\Gamma_h^\sigma} e_h^\sigma) \hat n_h^{*,\sigma} 
\d\theta \d\sigma \quad\, \mbox{(order of integration is changed)} \\
&=
\int_0^1 (1-\sigma) \frac{\d}{\d\sigma } \int_{\Gamma_h^\sigma}  H_h^{*,\sigma} e_h^\sigma\cdot (\nabla_{\Gamma_h^\sigma} e_h^\sigma) \hat n_h^{*,\sigma} \d\sigma
 \\
&=
\int_0^1 (1-\theta) \frac{\d}{\d\theta } \int_{\Gamma_h^\theta}  H_h^{*,\theta} e_h^\theta\cdot (\nabla_{\Gamma_h^\theta} e_h^\theta) \hat n_h^{*,\theta} 
\d\theta 
\quad\,\,\,\mbox{($\sigma$ is changed to $\theta$)} \\
&=
\int_0^1 (1-\theta) \int_{\Gamma_h^\theta}  
\Big( H_h^{*,\theta} e_h^\theta\cdot \partial_\theta^{\bullet}(\nabla_{\Gamma_h^\theta} e_h^\theta) \hat n_h^{*,\theta} 
 + (\nabla_{\Gamma_h^\theta} \cdot e_h^\theta) H_h^{*,\theta} e_h^\theta\cdot (\nabla_{\Gamma_h^\theta} e_h^\theta) \hat n_h^{*,\theta} 
\Big)
\d\theta ,
\end{align*} 
where the last equality uses the facts that $\partial_\theta^{\bullet} H_h^{*,\theta} = \partial_\theta^{\bullet} e_h^{\theta} =\partial_\theta^{\bullet} n_h^{*,\theta} =0$ and the surface $\Gamma_h^\theta$ moves with velocity $e_h^\theta$ with respect to $\theta$. 
By substituting the identity (cf. \cite[Lemma~2.6]{DziukKronerMuller}),  
\begin{align}
\label{eq:mat grad interchange-2}
\partial_\theta^\bullet \big( \nabla_{\Gamma_h^\theta} e_h^{\theta} \big) 
=
	\nabla_{\Gamma_h^\theta} \partial_\theta^\bullet e_h^{\theta} 
	-  \Big( \nabla_{\Gamma_h^\theta} e_h^\theta - \hat\nu_h^\theta (\hat\nu_h^\theta)^T (\nabla_{\Gamma_h^\theta} e_h^\theta)^T \Big) \nabla_{\Gamma_h^\theta}e_h^{\theta}  \notag \\
=
	- \Big( \nabla_{\Gamma_h^\theta} e_h^\theta - \hat\nu_h^\theta (\hat\nu_h^\theta)^T (\nabla_{\Gamma_h^\theta} e_h^\theta)^T \Big) \nabla_{\Gamma_h^\theta}e_h^{\theta} 	\quad\mbox{(since $\partial_\theta^\bullet e_h^{\theta} =0$)} 
\end{align}
into the above expression of $J_6$, we obtain 
\begin{align}\label{estimate-J6}
J_6
&\le
\int_0^1  c\|\nabla_{\Gamma_h^\theta} e_h^\theta\|_{L^\infty(\Gamma_h^\theta)}
\|\nabla_{\Gamma_h^\theta} e_h^\theta\|_{L^2(\Gamma_h^\theta)}
\|e_h^\theta\|_{L^2(\Gamma_h^\theta)} \d\theta \notag \\
&\le
\int_0^1 ch^2 
\|\nabla_{\Gamma_h^\theta} e_h^\theta\|_{L^2(\Gamma_h^\theta)}
\|e_h^\theta\|_{L^2(\Gamma_h^\theta)} \d\theta
\quad\mbox{(estimate \eqref{eq:assumed bounds - W1infty} is used)} \notag \\
&\le
\int_0^1 ch\|e_h^\theta\|_{L^2(\Gamma_h^\theta)}^2 \d\theta 
\quad\mbox{(inverse inequality)}\notag \\
&=
ch\|\bfe\|_{\bfM(\bfx^*)}^2 . 
\end{align}

Let ${\rm id}_{\Gamma_h^*}$ be the restriction of the identity function to the surface $\Gamma_h^*$, and note that the surface $\Gamma_h^\theta$ has parametrization ${\rm id}_{\Gamma_h^*}+\theta e_h^{*}$ defined on $\Gamma_h^*$. Hence, $\Gamma_h^\theta$ has parametrization ${\rm id}_{\Gamma_h^*}^l+\theta e_h^{*,l}$ defined on $\Gamma$. Let $\phi$ be a local parametrization of the surface $\Gamma$ in a chart, and let $\widetilde\phi =({\rm id}_{\Gamma_h^*}^l+\theta e_h^{*,l})  \circ \phi = {\rm id}_{\Gamma_h^*}^l\circ\phi + (\theta e_h^{*,l})  \circ \phi $. Then 
\begin{align*}
&\hat n_h^\theta \circ \widetilde\phi = \frac{ \partial_1 \widetilde\phi \times \partial_2\widetilde\phi}{|\partial_1 \widetilde\phi \times \partial_2\widetilde\phi|} 
\quad\mbox{and}\quad
\hat n_h^{*,\theta} \circ \widetilde\phi
=\hat n_h^{*}\circ {\rm id}_{\Gamma_h^*}^l\circ\phi 
=  \frac{ \partial_1({\rm id}_{\Gamma_h^*}^l\circ\phi) \times \partial_2({\rm id}_{\Gamma_h^*}^l\circ\phi)}{|\partial_1({\rm id}_{\Gamma_h^*}^l\circ\phi)\times \partial_2({\rm id}_{\Gamma_h^*}^l\circ\phi)|} . 
\end{align*}
Since the exact surface is non-degenerate, we have $c_1\le |\partial_1\phi \times \partial_2\phi|\le c_2$. Hence, 
\begin{align*}
|\hat n_h^\theta \circ \widetilde\phi - \hat n_h^{*,\theta} \circ \widetilde\phi|
\le 
c| \partial_1 (\widetilde\phi-{\rm id}_{\Gamma_h^*}^l\circ\phi)| + c| \partial_2 (\widetilde\phi-{\rm id}_{\Gamma_h^*}^l\circ\phi)| 
&\le
c \theta |\nabla (e_h^{*,l} \circ \phi) | \\
&\le
c \theta |(\nabla_{\Gamma} e_h^{*,l}) \circ \phi | |\nabla \phi | .
\end{align*}
This implies that 
\begin{align}\label{ntheta-nstar}
\|\hat n_h^\theta - \hat n_h^{*,\theta} \|_{L^2(\Gamma_h^\theta)} 
\le
c \|\nabla_{\Gamma} e_h^{*,l} \|_{L^2(\Gamma)} 
\le
c \|\nabla_{\Gamma_h^*} e_h^{*} \|_{L^2(\Gamma_h^*)} 
&\le
ch^{-1} \|e_h^{*} \|_{L^2(\Gamma_h^*)} , 
\end{align}
where the second to last inequality is obtained from \eqref{grad-Gamma-Gamma_h}. 
By using the estimate above, we have 
\begin{align}\label{estimate-J7}
J_7
&=\int_0^1 
\int_{\Gamma_h^\theta}  H_h^{*,\theta} e_h^\theta\cdot \nabla_{\Gamma_h^\theta} e_h^\theta \cdot (\hat n_h^\theta-\hat n_h^{*,\theta}) \,  
\d\theta \notag \\
&\le
\int_0^1  c\|e_h^\theta\|_{L^2(\Gamma_h^\theta)}
\|\nabla_{\Gamma_h^\theta} e_h^\theta\|_{L^\infty(\Gamma_h^\theta)} 
\|\hat n_h^\theta-\hat n_h^{*,\theta}\|_{L^2(\Gamma_h^\theta)} \d\theta \notag \\
&\le
\int_0^1 c\|e_h^\theta\|_{L^2(\Gamma_h^\theta)}
\|\nabla_{\Gamma_h^\theta} e_h^\theta\|_{L^\infty(\Gamma_h^\theta)} 
h^{-1} \|e_h^\theta\|_{L^2(\Gamma_h^\theta)} \d\theta 
\quad\mbox{(estimate \eqref{ntheta-nstar} is used)} \notag \\
&\le
\int_0^1 ch \|e_h^\theta\|_{L^2(\Gamma_h^\theta)}^2 \d\theta  \quad\mbox{(estimate \eqref{eq:assumed bounds - W1infty} is used)}  \notag \\
&\le ch \|\bfe\|_{\bfM(\bfx^*)}^2 .
\end{align}

Finally, 
\begin{align}
J_8
&=-
\int_0^1 
\int_{\Gamma_h^\theta}  H_h^{*,\theta} e_h^\theta\cdot (\nabla_{\Gamma_h^\theta} e_h^\theta) \hat n_h^\theta \,\d\theta \notag \\
&\le
c\int_0^1 \|e_h^\theta\|_{L^2(\Gamma_h^\theta)}
\|(\nabla_{\Gamma_h^\theta} e_h^\theta)  \hat n_h^\theta\|_{L^2(\Gamma_h^\theta)} \d\theta \notag \\
&\le
c\|\bfe\|_{\bfM(\bfx^*)} 
\bigg(\int_0^1 \int_{\Gamma_h^\theta} |(\nabla_{\Gamma_h^\theta} e_h^\theta) \hat n_h^\theta|^2 \d\theta\bigg)^{\frac12}  
\end{align}

Substituting the estimates of $J_m$, $m=1,\dots,8$, into \eqref{2nd-term-2}, we have 
\begin{align}\label{2nd-term-2-f}
-({\bf M}(\bfx )-{\bf M}(\bfx ^*))\dot \bfx ^*  \cdot \bfe  
&\le 
c\epsilon^{-1}\|\bfe\|_{\bfM(\bfx^*)}^2
+
\epsilon \int_0^1 \int_{\Gamma_h^\theta} |(\nabla_{\Gamma_h^\theta} e_h^\theta) \hat n_h^\theta|^2 \d\theta . 
\end{align} 
\begin{remark}
\upshape
The estimates \eqref{1st-term} and \eqref{2nd-term-2-f} together imply the result \eqref{mass-trick} mentioned in the introduction section. 
\end{remark}

Then, substituting \eqref{1st-term}--\eqref{3rd-term} and \eqref{2nd-term-2-f} into \eqref{FEM-Error-Eq-3}, we obtain 
\begin{align}\label{FEM-Error-Eq-4}
&\frac{\d}{\d t} \|\bfe\|_{\bfM(\bfx)}^2 
+ 2 \int_0^1 \int_{\Gamma_h^\theta} |(\nabla_{\Gamma_h^\theta} e_h^\theta) \hat n_h^\theta|^2 \d\theta \notag \\
&\le
ch^{2k-2} 
+ c\epsilon^{-1}\|\bfe\|_{\bfM(\bfx^*)}^2
+
2 \epsilon \int_0^1 \int_{\Gamma_h^\theta} |(\nabla_{\Gamma_h^\theta} e_h^\theta) \hat n_h^\theta|^2 \d\theta , 
\end{align} 
where $\epsilon$ can be an arbitrary positive number between $0$ and $1$. 
By choosing $\epsilon=\frac12$ and integrating the inequality above in time, we have 
\begin{align}\label{FEM-Error-Eq-5}
&\|\bfe(s)\|_{\bfM(\bfx)}^2 
+ \int_0^s \int_0^1 \int_{\Gamma_h^\theta} |(\nabla_{\Gamma_h^\theta} e_h^\theta) \hat n_h^\theta|^2 \d\theta \d t
\le
ch^{2k-2} 
+ c\int_0^s \|\bfe(t)\|_{\bfM(\bfx^*)}^2 \d t ,
\end{align} 
which holds for all $s\in(0,t^*]$. 
Since $\|\bfe(s)\|_{\bfM(\bfx)}$ is equivalent to $\|\bfe(s)\|_{\bfM(\bfx^*)}$, as explained in \eqref{norm-equiv}, applying Gronwall's inequality yields 
\begin{align*}
&\max_{t\in[0,t^*]} \|\bfe\|_{\bfM(\bfx^*)}^2 
+ \int_0^{t^*} \int_0^1 \int_{\Gamma_h^\theta} |(\nabla_{\Gamma_h^\theta} e_h^\theta) \hat n_h^\theta|^2 \d\theta \d t 
\le
ch^{2k-2}  .
\end{align*} 
Hence, 
\begin{align}\label{L2-eh}
&
\max_{t\in[0,t^*]} \| e_h(\cdot,t)\|_{L^2(\Gamma_h[\bfx^*(t)])} =
\max_{t\in[0,t^*]} \|\bfe\|_{\bfM(\bfx^*)}
\le
ch^{k-1}  .
\end{align} 
When $k\ge 6$ and $h$ sufficiently small, this implies 
\begin{align}\label{conclude-h4}
&\max_{t\in[0,t^*]} \| e_h(\cdot,t)\|_{L^2(\Gamma_h[\bfx^*(t)])} =
\max_{t\in[0,t^*]} \|\bfe\|_{\bfM(\bfx^*)}
\le
\frac12 h^{4}  .
\end{align} 
If $t^*<T$ then the inequality above furthermore implies that the solution can be extended to time $t^*+\epsilon_h$ for some sufficiently small $\epsilon_h$ such that \eqref{eq:assumed bounds - 2} holds. The maximality of $t^*$ for \eqref{eq:assumed bounds - 2} implies that $t^*=T$. 

Hence, \eqref{L2-eh} holds with $t^*=T$. This also implies, via inverse inequality, 
\begin{align*} 
\max_{t\in[0,T]} 
\| e_h(\cdot,t) \|_{L^\infty (\Gamma_h[\bfx^*(t)])} 
\le 
ch^{k-2} . 
\end{align*} 
This proves \eqref{x-xstar}. Since $X_h-X_h^*=e_h(\cdot,t)\circ X_h^*$, \eqref{L2-eh} also implies 
$$
\|X_h-X_h^*\|_{L^2(\Gamma_h(\bfx^0))} \le ch^{k-1} . 
$$
Lifting this onto $\Gamma_0$ yields 
$$
\|X_h^l-X_h^{*,l}\|_{L^2(\Gamma^0)} \le ch^{k-1} . 
$$
Then, using the triangle inequality and the interpolation error estimate \eqref{interpl-error}, we obtain
$$
\|X_h^l-X\|_{L^2(\Gamma^0)}
\le
\|X_h^l-X_h^{*,l}\|_{L^2(\Gamma^0)}
+\|X_h^{*,l}-X\|_{L^2(\Gamma^0)}
\le ch^{k-1} . 
$$
This completes the proof of Theorem \ref{MainTHM}. 
\hfill\endproof

\section{Proof of Lemma \ref{Lemma:key} }
\label{Section:Proof Lemma}

In this section we prove Lemma \ref{Lemma:key}, which is used in the proof of Theorem \ref{MainTHM}. Note that $\Gamma_h^\theta$ is the boundary of a bounded Lipschitz domain. We first prove the result for a smooth surface and then extend it to a general Lipschitz surface through approximating it by smooth surfaces. 

\begin{proposition}\label{prop:int_Gamma_star}
Let $\Gamma_{\!\star} $ be a bounded, closed and smooth surface and let $e\in H^1(\Gamma_{\!\star} )^3$. Then 
\begin{align}\label{int_Gamma_star}
\int_{\Gamma_{\!\star} } 
\Big[ {\rm tr}(\nabla_{\Gamma_{\!\star} }e)^2 - {\rm tr}(\nabla_{\Gamma_{\!\star} }e\nabla_{\Gamma_{\!\star} }e) \Big]
=0 . 
\end{align}
\end{proposition}
\begin{proof}
We denote $\nabla_{\Gamma_{\!\star} } f = (\D_1f,\D_2f,\D_3f)^T$ and use the following formula of integration by parts (cf. \cite[Definition 2.11]{DziukElliott_acta}) 
\begin{align}\label{int-by-parts}
\int_{\Gamma_{\!\star} } f \D_i\varphi 
= - \int_{\Gamma_{\!\star} } \D_i f \varphi 
+\int_{\Gamma_{\!\star} } f \varphi Hn_i .
\end{align}
If $e =(e^1,e^2,e^3)^T \in H^2(\Gamma_{\!\star} )^3$, then 
\begin{align*}
&\int_{\Gamma_{\!\star} } 
\Big[ {\rm tr}(\nabla_{\Gamma_{\!\star} }e)^2 - {\rm tr}(\nabla_{\Gamma_{\!\star} }e\nabla_{\Gamma_{\!\star} }e) \Big] \\
&=\int_{\Gamma_{\!\star} } 
\bigg[  (\D_1e^1+ \D_2e^2+\D_3e^3)^2
-  \D_ie^j \D_je^i \bigg]  \\
&=\int_{\Gamma_{\!\star} } (\D_1e^1+ \D_2e^2+\D_3e^3)^2 \\
&\quad\,
-  \int_{\Gamma_{\!\star} } (|\D_1e^1|^2 + |\D_2e^2|^2 + |\D_3e^3|^2+2\D_1e^2\D_2e^1 +2 \D_1e^3\D_3e^1+2\D_2e^3\D_3e^2)  \\
&= \int_{\Gamma_{\!\star} } 2\bigg[ 
(\D_1e^1\D_2e^2 -\D_1e^2\D_2e^1)
 + (\D_1e^1\D_3e^3 - \D_1e^3\D_3e^1)
 + (\D_2e^2\D_3e^3 - \D_2e^3\D_3e^2)
\bigg] .
\end{align*} 
By using \eqref{int-by-parts} and the formula (cf. \cite[Lemma2.6]{DziukElliott_acta}): 
$$
\D_i\D_j u = 
\D_j \D_i  u
+ Hn_i \D_ju - Hn_j \D_iu ,
$$
we have 
\begin{align*}
&\int_{\Gamma_{\!\star} } (\D_ie^i\D_je^j -\D_ie^j\D_je^i) \\
& = \int_{\Gamma_{\!\star} } (-e^i\D_i\D_je^j + H n_i e^i\D_je^j -\D_ie^j\D_je^i) \\
& = \int_{\Gamma_{\!\star} } (-e^i\D_j\D_ie^j -  Hn_i  e^i \D_je^j  + Hn_j e^i \D_ie^j   + H n_i e^i\D_je^j -\D_ie^j\D_je^i) \\
& = \int_{\Gamma_{\!\star} } (\D_je^i\D_ie^j -  Hn_j e^i  \D_ie^j  - Hn_i  e^i \D_je^j + Hn_j e^i \D_ie^j   + H n_i   e^i \D_je^j -\D_ie^j\D_je^i) \\
& = 0 .
\end{align*} 
This proves \eqref{int_Gamma_star} for $e\in H^2(\Gamma_{\!\star})^3$. Since $H^2(\Gamma_{\!\star})^3$ is dense in $H^1(\Gamma_{\!\star})^3$, it follows that \eqref{int_Gamma_star} also holds for $e\in H^1(\Gamma_{\!\star})^3$. 
\hfill\end{proof}

By using Proposition \ref{prop:int_Gamma_star}, we prove the following result, which implies Lemma \ref{Lemma:key}. 
 
\begin{proposition}\label{Lemma:int_Lipscthiz_Gamma}
If $\Gamma_{\!\star}$ is the boundary of a bounded Lipschitz domain $\Omega$, then for $e\in H^1(\Gamma_{\!\star} )^3$ the following identity holds:
\begin{align}\label{int_Lipscthiz_Gamma}
\int_{\Gamma_{\!\star} } 
\Big[ {\rm tr}(\nabla_{\Gamma_{\!\star} }e)^2 - {\rm tr}(\nabla_{\Gamma_{\!\star} }e\nabla_{\Gamma_{\!\star} }e) \Big]
=0 . 
\end{align}
\end{proposition}
\begin{proof}
In the following, we show that there exists a sequence of smooth functions $\tilde w^n\in C^{\infty}(\R^3)^3$ such that $\tilde w^n$ converges to $e$ in $H^1(\Gamma_{\!\star})$ as $n\rightarrow\infty$, and a sequence of smooth domains $\Omega_m$ with smooth boundary $\Gamma_{\!\star}^m$ such that $\Gamma_{\!\star}^m\rightarrow \Gamma_{\!\star}$ as $m\rightarrow\infty$. By using the result of Proposition \ref{prop:int_Gamma_star}, we have
\begin{align}\label{key-W1infty-R3-m}
\int_{\Gamma_{\!\star}^m } 
\Big[ {\rm tr}(\nabla_{\Gamma_{\!\star}^m }\tilde w^n)^2 - {\rm tr}(\nabla_{\Gamma_{\!\star}^m }\tilde w^n\nabla_{\Gamma_{\!\star}^m }\tilde w^n) \Big]
=0 . 
\end{align}
By taking $m\rightarrow \infty$ in the equality above, we shall prove the following result: 
\begin{align}\label{key-W1infty-R3}
\int_{\Gamma_{\!\star}} 
\Big[ {\rm tr}(\nabla_{\Gamma_{\!\star}}\tilde w^n)^2 - {\rm tr}(\nabla_{\Gamma_{\!\star}}\tilde w^n\nabla_{\Gamma_{\!\star}}\tilde w^n) \Big]
=0 . 
\end{align}
This would prove the desired result for the smooth function $\tilde w^n\in W^{1,\infty}(\R^3)^3$. Since $\tilde w^n\rightarrow e$ in $H^1(\Gamma_{\!\star})$, letting $n\rightarrow 0$ in \eqref{key-W1infty-R3} yields the desired result \eqref{int_Lipscthiz_Gamma}. 

First, we consider a partition of unity $\phi_j\in C^\infty_0(\R^3)$, $j=1,\dots,J$, such that 
$\sum_{j=1}^J \phi_j = 1 $ in a neighborhood of $\Gamma_\star$ and each $\phi_j$ has compact support in an open ball $B_j$ in which the surface $\Gamma_\star\cap B_j$ can be represented by a Lipschitz graph after a rotation $Q_j$:
\begin{align}
&\Gamma_\star\cap B_j=\{Q_jx : x_3=\varphi_j(x_1,x_2),\,\,\,(x_1,x_2) \in D_j \} ,\\
&B_j\cap \Omega\subset \{Q_jx : x_3>\varphi_j(x_1,x_2),\,\,\,(x_1,x_2) \in D_j \} , \\
&B_j\backslash\overline \Omega\subset \{Q_jx : x_3<\varphi_j(x_1,x_2),\,\,\,(x_1,x_2) \in D_j \} ,
\end{align}
where $\varphi_j$ is a Lipschitz continuous function on $D_j$, which is a bounded domain in $\R^2$. Hence, 
$$
e=\sum_{j=1}^J e\phi_j\quad\mbox{on}\quad \Gamma_\star .
$$ 
For the Lipschitz domain $\Omega$, there exists a sequence of domains $\Omega_m$, $m=1,2,\dots$, with smooth boundary $\Gamma_{\!\star}^m$ such that $\Gamma_{\!\star}^m\rightarrow \Gamma_{\!\star}$ as $m\rightarrow\infty$ in the following sense (see \cite[Theorem 5.1]{Doktor-1976}): 
\begin{align}\label{define-varphij-m}
&\Gamma_\star^m\cap B_j=\{Q_jx : x_3=\varphi_j^m(x_1,x_2),\,\,\,(x_1,x_2) \in D_j \} ,
\end{align} 
where $\varphi_j^m$, $m=1,2,\dots$, is a sequence of functions converging to $\varphi_j$ strongly in both $L^\infty(D_j)$ and $W^{1,p}(D_j)$ for all $p\in[1,\infty)$, and $\nabla\varphi_j^m$ converges to $\nabla\varphi_j$ weakly$^*$ in $L^\infty(D_j)^3$ ($\nabla\varphi_j^m$ is bounded in $L^\infty(D_j)^3$ as $m\rightarrow \infty$). 

Next, on the two-dimensional region $D_j$, we define $\Phi_j(x_1,x_2)=(x_1,x_2,\varphi_j(x_1,x_2))^T \in\R^3$ and 
\begin{align}\label{definition-wj}
w_j(x_1,x_2) = (e\phi_j) \circ (Q_j\Phi_j)(x_1,x_2) \quad\mbox{for}\,\,\, (x_1,x_2) \in D_j . 
\end{align}
Then $Q_j\Phi_j:D_j\rightarrow \Gamma_\star\cap B_j$ is a parametrization of $\Gamma_\star\cap B_j$ and $w_j\in H^1_0(D_j)^3$. We can approximate $w_j$ in $H^1(D_j)^3$ by a sequence of smooth functions $w_j^n\in C^{\infty}(\R^2)^3$ with compact supports inside $D_j$. These functions have natural extensions to $\overline  w_j^n\in C^{\infty}(\R^3)^3$, i.e., 
\begin{align}\label{definition-wj2}
\overline w_j^n (x_1,x_2,x_3)=w_j^n (x_1,x_2) \chi_\alpha(x_3) \quad\mbox{for}\,\,\, (x_1,x_2,x_3)\in\R^3 , 
\end{align}
where $\chi_\alpha(x_3)$ is a one-dimensional smooth cut-off function which satisfies 
\begin{align}\label{property-chi}
\chi_\alpha(0)=1,
\quad\chi_\alpha'(0)=0\quad\mbox{and}\quad \chi_\alpha(x_3)=0\quad\mbox{for}\quad
|x_3|>\alpha. 
\end{align}
Then we can define a smooth function $\hat w_j^n\in  C^{\infty}(\R^3)^3$ (with compact support in $B_j$) that approximates $e\phi_j$ in $H^1(\Gamma_\star\cap B_j)$, i.e., 
\begin{align}\label{definition-wj3}
\tilde w_j^n(Q_jx)
=\overline w_j^n(x_1,x_2,x_3 - \varphi_j^n(x_1,x_2)) \quad\mbox{for}\,\,\, (x_1,x_2,x_3)^T\in\R^3 . 
\end{align}
By choosing a sufficiently small $\alpha$, the extended functions $\tilde w_j^n\in C^{\infty}(\R^3)^3$ have compact supports in $B_j$. 
Since $Q_j\Phi_j:D_j\rightarrow \Gamma_\star\cap B_j$ is a parametrization of $\Gamma_\star\cap B_j$, it follows that ``$\tilde w_j^n$ converges to $e\phi_j$ in $H^1(\Gamma_\star\cap B_j)$'' if and only if ``$\tilde w_j^n\circ (Q_j\Phi_j)$ converges to $(e\phi_j)\circ(Q_j\Phi_j)$ in $H^1(D_j)$''. 
In view of the definitions in \eqref{definition-wj}--\eqref{definition-wj2} and \eqref{definition-wj3}, we have 
\begin{align}\label{tilde-wjn-ephij}
&\tilde w_j^n \circ (Q_j\Phi_j)(x_1,x_2) - (e\phi_j) \circ (Q_j\Phi_j)(x_1,x_2) \notag\\
&= w_j^n (x_1,x_2) \chi_\alpha(\varphi_j(x_1,x_2) - \varphi_j^n(x_1,x_2))  - w_j(x_1,x_2) \notag\\
&= w_j^n (x_1,x_2) [\chi_\alpha(\varphi_j(x_1,x_2) - \varphi_j^n(x_1,x_2))-1] 
+ [w_j^n (x_1,x_2) - w_j(x_1,x_2) ] . 
\end{align}
Since $\varphi_j^n$ converges to $\varphi_j$ in $L^\infty(D_j)\cap W^{1,p}(D_j)$ as $n\rightarrow\infty$ for arbitrary $p\in[1,\infty)$ (see the statement below \eqref{define-varphij-m}), and $w_j^n$ converges to $w_j$ in $H^{1}(D_j)\hookrightarrow L^p(D_j)$ for all $p\in[1,\infty)$ (this is how $w_j^n$ is defined), from \eqref{tilde-wjn-ephij} it is straightforward to verify that $\tilde w_j^n\circ (Q_j\Phi_j)$ converges to $(e\phi_j)\circ(Q_j\Phi_j)$ in $H^1(D_j)$. As a result, $\tilde w_j^n$ converges to $e\phi_j$ in $H^1(\Gamma_\star\cap B_j)$. 
Therefore, 
$$
\tilde w^n=\sum_{j=1}^J \tilde w_j^n ,\quad n=1,2,\dots,
$$
is a sequence of functions in $C^{\infty}(\R^3)^3$ that converges to $e=\sum_{j=1}^Je\phi_j$ in $H^1(\Gamma_{\!\star})$ as $n\rightarrow\infty$.

Finally, we prove that taking $m\rightarrow \infty$ in \eqref{key-W1infty-R3-m} would yield \eqref{key-W1infty-R3}. This would complete the proof of Proposition \ref{Lemma:int_Lipscthiz_Gamma}. 
To this end, we consider the decomposition 
\begin{align}\label{key-W1infty-R3-01} 
& \int_{\Gamma_{\!\star}^m } 
\Big[ {\rm tr}(\nabla_{\Gamma_{\!\star}^m }\tilde w^n)^2 - {\rm tr}(\nabla_{\Gamma_{\!\star}^m }\tilde w^n\nabla_{\Gamma_{\!\star}^m }\tilde w^n) \Big] \notag\\
&=  \sum_{j=1}^J \int_{\Gamma_{\!\star}^m\cap B_j} {\rm tr}(\nabla_{\Gamma_{\!\star}^m }\tilde w^n)^2 \phi_j 
- \sum_{j=1}^J \int_{\Gamma_{\!\star}^m\cap B_j} {\rm tr}(\nabla_{\Gamma_{\!\star}^m }\tilde w^n\nabla_{\Gamma_{\!\star}^m }\tilde w^n) \phi_j 
\end{align}
and prove the following two results: 
\begin{align}
&\lim_{m\rightarrow 0} \int_{\Gamma_{\!\star}^m\cap B_j} {\rm tr}(\nabla_{\Gamma_{\!\star}^m }\tilde w^n)^2 \phi_j 
= \int_{\Gamma_{\!\star}\cap B_j} {\rm tr}(\nabla_{\Gamma_{\!\star} }\tilde w^n)^2 \phi_j 
&&\mbox{for every $j$} , \label{key-W1infty-R3-1}\\
&\lim_{m\rightarrow 0} \int_{\Gamma_{\!\star}^m\cap B_j} {\rm tr}(\nabla_{\Gamma_{\!\star}^m }\tilde w^n\nabla_{\Gamma_{\!\star}^m }\tilde w^n) \phi_j 
= \int_{\Gamma_{\!\star}\cap B_j}
{\rm tr}(\nabla_{\Gamma_{\!\star} }\tilde w^n\nabla_{\Gamma_{\!\star} }\tilde w^n) \phi_j 
&&\mbox{for every $j$} .  \label{key-W1infty-R3-1-2}
\end{align}

Let $\Phi_j^m(x_1,x_2)=(x_1,x_2,\varphi_j^m(x_1,x_2))^T \in\R^3$. Then $\Phi_j^m$ is a parametrization of the surface $\Gamma_\star^m\cap B_j$ after a rotation by $Q_j$. By using this parametrization, the left-hand side of \eqref{key-W1infty-R3-1} can be written as 
\begin{align}\label{integral-m} 
&\int_{\Gamma_{\!\star}^m\cap B_j} {\rm tr}(\nabla_{\Gamma_{\!\star}^m }\tilde w^n)^2 \phi_j \\
&=  \int_{D_j}  {\rm tr} \Big(\sum_{i,\ell=1}^2 g^{i\ell}(\nabla\Phi_j^m) \frac{\partial \tilde w^n(Q_j\Phi_j^m)}{\partial x_\ell} \otimes \partial_{x_i}\Phi_j^m \Big)^2 (\phi_j\circ \Phi_j^m) 
\sqrt{1+|\nabla \varphi_j^m|^2} \, \d x_1\d x_2 . \notag 
\end{align}
where $g^{i\ell}(\nabla\Phi_j^m)$ is the inverse matrix of the Riemannian metric tensor $g_{i\ell}(\nabla\Phi_j^m)$, i.e., 
$$
g_{i\ell}(\nabla\Phi_j^m) =\partial_{x_i}\Phi_j^m\cdot \partial_{x_\ell}\Phi_j^m , \quad 
i,\ell=1,2 .
$$ 
Since $\Phi_j^m$ converges to $\Phi_j$ in $L^\infty(D_j)\cap W^{1,p}(D_j)$ as $m\rightarrow\infty$ for all $p\in[1,\infty)$, it follows that $g_{i\ell}(\nabla\Phi_j^m)$ converges to $g_{i\ell}(\nabla\Phi_j)$ in $L^p(D_j)$ for all $p\in[1,\infty)$. Furthermore, since 
$$
\det(g_{i\ell}(\nabla\Phi_j^m)) 
=1+|\nabla\varphi_j^m|^2 
$$
is bounded from both below and above (because $\nabla\varphi_j^m$ is bounded in $L^\infty(D_j)^3$ as $m\rightarrow \infty$), it follows that the inverse matrix $g^{i\ell}(\nabla\Phi_j^m)$ also converges, i.e.,  
\begin{align}\label{converg-gil} 
\mbox{$g^{i\ell}(\nabla\Phi_j^m)$ converges to $g^{i\ell}(\nabla\Phi_j)$ in $L^p(D_j)$ for all $p\in[1,\infty)$
as $m\rightarrow \infty$.} 
\end{align}
Note that  
\begin{align*} 
\frac{\partial \tilde w^n(Q_j\Phi_j^m(x_1,x_2))}{\partial x_\ell} 
&= \Big( \frac{\partial \tilde w^n}{\partial x_q} \circ (Q_j\Phi_j^m)(x_1,x_2) \Big)
Q_{j,q}\frac{\partial\Phi_j^m(x_1,x_2)}{\partial x_\ell} , \\
\frac{\partial \tilde w^n(Q_j\Phi_j(x_1,x_2))}{\partial x_\ell} 
&= \Big( \frac{\partial \tilde w^n}{\partial x_q} \circ (Q_j\Phi_j)(x_1,x_2) \Big)
Q_{j,q}\frac{\partial\Phi_j(x_1,x_2)}{\partial x_\ell} , 
\end{align*} 
where $Q_{j,q}$ denotes the $q$th row of $Q_j$. Since $\displaystyle \frac{\partial \tilde w^n}{\partial x_q}\in C^\infty(\R^3)^3$ for fixed $n$ and $\Phi_j^m$ converges to $\Phi_j$ in $L^\infty(D_j)\cap W^{1,p}(D_j)$ for all $p\in[1,\infty)$ as $m\rightarrow\infty$, it follows that 
\begin{align}\label{converg-dwn} 
\mbox{$\displaystyle \frac{\partial [\tilde w^n\circ (Q_j\Phi_j^m)]}{\partial x_\ell} $ converges to $\displaystyle\frac{\partial [\tilde w^n\circ (Q_j\Phi_j)]}{\partial x_\ell} $ in $L^p(D_j)$ for all $p\in[1,\infty)$ as $m\rightarrow \infty$.} 
\end{align}
Since $\phi_j$ is smooth and $\Phi_j^m$ converges to $\Phi_j$ in $L^\infty(D_j)$ as $m\rightarrow\infty$, it follows that 
\begin{align}\label{converg-phij-varphij} 
\mbox{$\phi_j\circ\Phi_j^m$ converges to $\phi_j\circ\Phi_j$ in $L^\infty(D_j)$ as $m\rightarrow \infty$. } 
\end{align}
Then, substituting \eqref{converg-gil}, \eqref{converg-dwn} and \eqref{converg-phij-varphij} into the right-hand side of \eqref{integral-m} and taking limit $m\rightarrow\infty$, we obtain \eqref{key-W1infty-R3-1}. 
The proof of \eqref{key-W1infty-R3-1-2} is similar and omitted. 

Substituting \eqref{key-W1infty-R3-1}--\eqref{key-W1infty-R3-1-2} into \eqref{key-W1infty-R3-01} yields the desired result \eqref{key-W1infty-R3}. This completes the proof of Proposition \ref{Lemma:int_Lipscthiz_Gamma}. 
\hfill\end{proof}

\section{Proof of the defect's estimate (\ref{defect-estimate})}\label{section:defect}

In this section we prove (\ref{defect-estimate}), which is used in the proof of Theorem \ref{MainTHM}. We rewrite equation \eqref{PDE-X-P} into  
\begin{align}\label{PDE-X-P-2}
&\partial_t^{\bullet}{\rm id} = \Delta_{\Gamma[X(\cdot,t)]}{\rm id} \quad\mbox{on}\,\,\,\Gamma[X(\cdot,t)],\,\,\,\forall\, t\in(0,T]. 
\end{align}
Let $w_h\in S_h(\Gamma_h[\bfx^*])$ be a finite element function on the interpolated surface $\Gamma_h[\bfx^*]$, and let $w_h^l\in H^1(\Gamma)$ be the lift of $w_h$ onto the exact surface $\Gamma=\Gamma[X(\cdot,t)]$. Then, testing \eqref{PDE-X-P-2} by $w_h^l$, we obtain
\begin{align}\label{PDE-weak}
&\int_{\Gamma} \partial_t^{\bullet}{\rm id} \cdot w_h^l 
+ \int_{\Gamma}  \nabla_{\Gamma}{\rm id}\cdot\nabla_{\Gamma} w_h^l = 0 \quad\forall\, w_h\in S_h(\Gamma_h[\bfx^*]). 
\end{align}
This can be furthermore written as
\begin{align}\label{PDE-weak-h}
&\int_{\Gamma_h^*} \partial_{t,h}^{\bullet}{\rm id} \cdot w_h
+ \int_{\Gamma_h^*}  \nabla_{\Gamma_h^*}{\rm id}\cdot\nabla_{\Gamma_h^*} w_h 
= \int_{\Gamma_h^*} d_h \cdot w_h , \quad\forall\, w_h\in S_h(\Gamma_h[\bfx^*]) , 
\end{align}
where $d_h\in S_h(\Gamma_h^*)$ is the unique finite element function determined by the relation 
\begin{align*} 
\int_{\Gamma_h^*} d_h \cdot w_h 
=&
\bigg( \int_{\Gamma_h^*} \partial_{t,h}^{\bullet}{\rm id} \cdot w_h
- \int_{\Gamma} \partial_t^{\bullet}{\rm id} \cdot w_h^l  \bigg) \\
&
+ \bigg( \int_{\Gamma_h^*}  \nabla_{\Gamma_h^*}{\rm id}\cdot\nabla_{\Gamma_h^*} w_h 
-\int_{\Gamma}  \nabla_{\Gamma}{\rm id}\cdot\nabla_{\Gamma} w_h^l \bigg) \\
=&\!: \mathcal{E}_1(w_h)+\mathcal{E}_2(w_h) . 
\end{align*}
In the matrix-vector form, \eqref{PDE-weak-h} can be equivalently written as 
\begin{align} 
&{\bf M}(\bfx^*)\dot \bfx^* + {\bf A}(\bfx^*)\bfx^*
= {\bf M}(\bfx^*) {\bf d} ,
\end{align}
with ${\bf d}$ being the nodal vector of the finite element function $d_h\in S_h(\Gamma_h^*)$. 

Note that $\partial_{t,h}^{\bullet}{\rm id}=v_h^*$ on $\Gamma_h^*$ and $\partial_{t}^{\bullet}{\rm id}=v$ on $\Gamma$, where $v_h^*$ and $v$ are the velocity of the surfaces $\Gamma_h^*$ and $\Gamma$, respectively. In particular, $v_h^*$ is the Lagrange interpolation of $v$. Hence, by using \eqref{interpl-error} and \eqref{delta_h-1},
\begin{align*} 
\mathcal{E}_1(w_h)
&=
\int_{\Gamma_h^*} v_h^* \cdot w_h
- \int_{\Gamma} v\cdot w_h^l \\
&=
\bigg( \int_{\Gamma_h^*} v_h^* \cdot w_h - \int_{\Gamma} v_h^{*,l} \cdot w_h^l \bigg) 
+ \int_{\Gamma} (v_h^{*,l} - v) \cdot w_h^l \\
&=
 \int_{\Gamma_h^*} (1-\delta_h) v_h^* \cdot w_h 
+ \int_{\Gamma} (v_h^{*,l} - v) \cdot w_h^l \\
&\le
ch^{k+1}\|v_h^*\|_{L^2(\Gamma_h^*)}\|w_h\|_{L^2(\Gamma_h^*)} 
+ch^{k+1}\|w_h^l\|_{L^2(\Gamma)} \\
&\le
ch^{k+1}\|w_h\|_{L^2(\Gamma_h^*)} .
\end{align*}
Let ${\rm id}_{\Gamma_h^*}$ and ${\rm id}_{\Gamma}$ be the identity function restricted to $\Gamma_h^*$ and $\Gamma$, respectively, and let ${\rm id}_{\Gamma_h^*}^l$ be the lifted function on $\Gamma$. Then  
\begin{align*} 
\mathcal{E}_2(w_h)
=&
\int_{\Gamma_h^*}  \nabla_{\Gamma_h^*}{\rm id}_{\Gamma_h^*}\cdot\nabla_{\Gamma_h^*} w_h 
-\int_{\Gamma}  \nabla_{\Gamma}{\rm id}_{\Gamma}\cdot\nabla_{\Gamma} w_h^l \\
=&
\bigg( \int_{\Gamma_h^*}  \nabla_{\Gamma_h^*}{\rm id}_{\Gamma_h^*}\cdot\nabla_{\Gamma_h^*} w_h 
-\int_{\Gamma}  \nabla_{\Gamma}{\rm id}_{\Gamma_h^*}^l\cdot\nabla_{\Gamma} w_h^l \bigg) 
+ \int_{\Gamma}  \nabla_{\Gamma}({\rm id}_{\Gamma_h^*}^l - {\rm id}_{\Gamma}) \cdot\nabla_{\Gamma} w_h^l \\
\le &
ch^{k+1}\|\nabla_{\Gamma_h^*}{\rm id}_{\Gamma_h^*}\|_{L^2(\Gamma_h^*)} \|\nabla_{\Gamma_h^*}w_h\|_{L^2(\Gamma_h^*)}
+ch^{k} \|\nabla_{\Gamma_h^*}w_h\|_{L^2(\Gamma_h^*)} \\
\le &
ch^{k}\|\nabla_{\Gamma_h^*}{\rm id}_{\Gamma_h^*}\|_{L^2(\Gamma_h^*)} \|w_h\|_{L^2(\Gamma_h^*)}
+ch^{k-1} \|w_h\|_{L^2(\Gamma_h^*)} ,
\end{align*}
where the second to last inequality again uses \cite[Lemma 5.2]{Kovacs2018}. 
This proves that 
$$
\bigg| \int_{\Gamma_h^*} d_h \cdot w_h\bigg|
\le
ch^{k-1} \|w_h\|_{L^2(\Gamma_h^*)} .
$$
In the matrix-vector form, this can be equivalently written as 
$$
|{\bf M}(\bfx^*) {\bf d} \cdot {\bf w} |
\le ch^{k-1} \|{\bf w}\|_{{\bf M}(\bfx^*)} 
$$
Hence, by choosing ${\bf w}={\bf d}$ in the inequality above, we obtain 
$$
\|{\bf d}\|_{{\bf M}(\bfx^*)} 
\le ch^{k-1} .
$$
This proves the defect's estimate \eqref{defect-estimate}. 
\hfill\endproof

\section{Concluding remarks}

The main contribution of this paper is the discovery of the structure \eqref{monotone} and its application to proving the convergence of Dziuk's semidiscrete FEM for mean curvature flow of closed surfaces with sufficiently high-order finite elements. 

The following additional difficulty would appear in the analysis of linearly implicit time discretisation: 
\begin{align}\label{monotone2}
\big({\bf A}(\bfx^{n-1})\bfx^n -{\bf A}(\bfx^{*,n-1})\bfx^{*,n}\big)\cdot(\bfx^n -\bfx^{*,n}) 
\end{align}
is no longer in the form of the left-hand side of \eqref{monotone} due to the shift of superscript indices. Hence, additional terms would appear in converting \eqref{monotone2} to the form of the left-hand side of \eqref{monotone}. Those additional terms may be bounded by using the approach in \cite{Li-2020-SINUM} under a certain grid-ratio condition.

It is straightforward to verify that both \eqref{idea-surface-2} and Proposition \ref{prop:int_Gamma_star} can be extended to higher dimensions, i.e., for mean curvature flow of $d$-dimensional hypersurfaces in $\R^{d+1}$  with $d\ge 2$. As a result, the monotone structure and the convergence proof can be generalised to this case. However, the monotone structure of mean curvature flow of two-dimensional surfaces in higher codimension is not obvious from the current proof, and therefore the convergence of evolving surface FEMs in this case still remains open. 

Convergence of Dziuk's semidiscrete FEM with low-order finite elements, as well as the parametric FEMs of Barrett, Garcke \& N{\"u}rnberg \cite{BGN2008,Barrett-2010-4}, remain open for mean curvature flow of closed surfaces. 
Efficient numerical methods for the non-divergence parabolic system constructed from DeTurck's trick in \cite{Elliott-Fritz-2017}, allowing singularity to appear in the numerical simulation of closed surfaces, is still challenging.   

\bigskip

\section*{Acknowledgement}
I would like to thank Prof. Christian Lubich for reading the manuscript and providing many valuable comments and suggestions.

\renewcommand{\refname}{\bf References\vspace{-5pt}}

\bibliographystyle{abbrv}
\bibliography{evolving_surface_literature}

\begin{thebibliography}{10}

\bibitem{2019-Barrett-Deckelnick-Nurnberg}
J.~W. Barrett, K.~Deckelnick, and R.~N{\"u}rnberg.
\newblock {A finite element error analysis for axisymmetric mean curvature
  flow}.
\newblock {\em arXiv:1911.05398}, 2019.

\bibitem{Barrett-Deckelnick-Styles-2017}
J.~W. Barrett, K.~Deckelnick, and V.~Styles.
\newblock {Numerical analysis for a system coupling curve evolution to reaction
  diffusion on the curve}.
\newblock {\em SIAM J. Numer. Anal.}, 55(2):1080--1100, 2017.

\bibitem{BGN2008}
J.~W. Barrett, H.~Garcke, and R.~N{\"u}rnberg.
\newblock On the parametric finite element approximation of evolving
  hypersurfaces in {$\R^3$}.
\newblock {\em J. Comput. Phys.}, 227(9):4281--4307, 2008.

\bibitem{Barrett-2010-4}
J.~W. Barrett, H.~Garcke, and R.~N{\"u}rnberg.
\newblock {Numerical approximation of gradient flows for closed curves in
  $\mathbb{R}^d$}.
\newblock {\em IMA J. Numer. Anal.}, 30(1):4--60, 2010.

\bibitem{deckelnick1995convergence}
K.~Deckelnick and G.~Dziuk.
\newblock Convergence of a finite element method for non-parametric mean
  curvature flow.
\newblock {\em Numer. Math.}, 72(2):197--222, 1995.

\bibitem{DeckelnickDziuk}
K.~Deckelnick and G.~Dziuk.
\newblock On the approximation of the curve shortening flow.
\newblock In {\em Calculus of variations, applications and computations
  ({P}ont-\`a-{M}ousson, 1994)}, volume 326 of {\em Pitman Res. Notes Math.
  Ser.}, pages 100--108. Longman Sci. Tech., Harlow, 1995.

\bibitem{deckelnick2000error-graph}
K.~Deckelnick and G.~Dziuk.
\newblock Error estimates for a semi-implicit fully discrete finite element
  scheme for the mean curvature flow of graphs.
\newblock {\em Interfaces Free Bound.}, 2(4):341--359, 2000.

\bibitem{DeckelnickDE2005}
K.~Deckelnick, G.~Dziuk, and C.~M. Elliott.
\newblock Computation of geometric partial differential equations and mean
  curvature flow.
\newblock {\em Acta Numerica}, 14:139--232, 2005.

\bibitem{Demlow2009}
A.~Demlow.
\newblock Higher--order finite element methods and pointwise error estimates
  for elliptic problems on surfaces.
\newblock {\em SIAM J. Numer. Anal.}, 47(2):805--807, 2009.

\bibitem{Doktor-1976}
P.~Doktor.
\newblock {Approximation of domains with Lipschitzian boundary}.
\newblock {\em Časopis pro pěstování matematiky}, 101(3):237--255, 1976.

\bibitem{Dziuk88}
G.~Dziuk.
\newblock Finite elements for the {B}eltrami operator on arbitrary surfaces.
\newblock {\em Partial differential equations and calculus of variations,
  Lecture Notes in Math., 1357, Springer, Berlin}, pages 142--155, 1988.

\bibitem{Dziuk90}
G.~Dziuk.
\newblock An algorithm for evolutionary surfaces.
\newblock {\em Numer. Math.}, 58(1):603--611, 1990.

\bibitem{Dziuk94}
G.~Dziuk.
\newblock Convergence of a semi-discrete scheme for the curve shortening flow.
\newblock {\em Math. Models Methods Appl. Sci.}, 4(4):589--606, 1994.

\bibitem{Dziuk1999}
G.~Dziuk.
\newblock Discrete anisotropic curve shortening flow.
\newblock {\em SIAM J. Numer. Anal.}, 36(6):1808--1830, 1999.

\bibitem{DziukElliott_ESFEM}
G.~Dziuk and C.~M. Elliott.
\newblock Finite elements on evolving surfaces.
\newblock {\em IMA J. Numer. Anal.}, 27(2):262--292, 2007.

\bibitem{DziukElliott_acta}
G.~Dziuk and C.~M. Elliott.
\newblock Finite element methods for surface {PDE}s.
\newblock {\em Acta Numerica}, 22:289--396, 2013.

\bibitem{Dziuk-Elliott-2013-MC}
G.~Dziuk and C.~M. Elliott.
\newblock {$L^2$-estimates for the evolving surface finite element method}.
\newblock {\em Math. Comp.}, 82(281):1--24, 2013.

\bibitem{DziukKronerMuller}
G.~Dziuk, D.~Kr\"{o}ner, and T.~M\"{u}ller.
\newblock Scalar conservation laws on moving hypersurfaces.
\newblock {\em Interfaces Free Bound.}, 15(2):203--236, 2013.

\bibitem{Ecker2012}
K.~Ecker.
\newblock {\em Regularity theory for mean curvature flow}.
\newblock Springer, 2012.

\bibitem{Elliott-Fritz-2017}
C.~M. Elliott and H.~Fritz.
\newblock {On approximations of the curve shortening flow and of the mean
  curvature flow based on the DeTurck trick}.
\newblock {\em IMA J. Numer. Anal.}, 37(2):543--603, 2017.

\bibitem{Kovacs2018}
B.~Kov\'{a}cs.
\newblock High-order evolving surface finite element method for parabolic
  problems on evolving surfaces.
\newblock {\em IMA J. Numer. Anal.}, 38(1):430--459, 2018.

\bibitem{Kovacs-Li-Lubich-2019}
B.~Kov{\'a}cs, B.~Li, and C.~Lubich.
\newblock {A convergent evolving finite element algorithm for mean curvature
  flow of closed surfaces}.
\newblock {\em Numer. Math.}, 143:797--853, 2019.

\bibitem{KLLP2017}
B.~Kov\'{a}cs, B.~Li, C.~Lubich, and C.~{Power Guerra}.
\newblock Convergence of finite elements on an evolving surface driven by
  diffusion on the surface.
\newblock {\em Numer. Math.}, 137(3):643--689, 2017.

\bibitem{Li-2020-SINUM}
B.~Li.
\newblock {Convergence of Dziuk's linearly implicit parametric finite element
  method for curve shortening flow}.
\newblock {\em SIAM J. Numer. Anal.}, 58(4):2315--2333, 2020.

\bibitem{Pozzi2007}
P.~Pozzi.
\newblock Anisotropic curve shortening flow in higher codimension.
\newblock {\em Math. Meth. Appl. Sci.}, 30(11):1243--1281, 2007.

\bibitem{Pozzi-Stinner-2017}
P.~Pozzi and B.~Stinner.
\newblock {Curve shortening flow coupled to lateral diffusion}.
\newblock {\em Numer. Math.}, 135:1171--1205, 2017.

\bibitem{2006-Rusu}
R.~Rusu.
\newblock {Numerische analysis für den Krümmungsfluß und den Willmorefluß}.
\newblock {\em PhD Thesis}, University of Freiburg, Freiburg, 2006.

\end{thebibliography}

\end{document}